\documentclass[preprint,12pt]{elsarticle}

\usepackage{amssymb}
\usepackage{amsthm}

\usepackage{lineno}

\usepackage{amsmath}

\usepackage{multirow}
\usepackage{ifthen}
\usepackage{subfig}
\usepackage{array}
\usepackage{booktabs}
\usepackage{lscape}

\usepackage{tikz}
\usetikzlibrary{shapes}
\usetikzlibrary{fit, calc}

\usepackage{rotating}

\usepackage{soul}
\definecolor{myyellow}{RGB}{245, 255, 189} 
\sethlcolor{myyellow}

\usepackage{ bbold }

\newcolumntype{M}[1]{>{\centering}m{#1}}

\newcommand{\tagarray}{%
\mbox{}\refstepcounter{equation}%
$(\theequation)$%
}

\graphicspath{{}}

\pgfdeclarelayer{background}
\pgfdeclarelayer{foreground}
\pgfsetlayers{background,main,foreground}

\tikzstyle{every picture}+=[remember picture]

\newtheorem{theorem}{Theorem}[section]
\newtheorem{lemma}[theorem]{Lemma}

\newenvironment{remark}[1][Remark]{\begin{trivlist}
\item[\hskip \labelsep {\bfseries #1}]}{\end{trivlist}}

\newcommand {\zdef}{\stackrel{def}{=}}
\newcommand {\tr}[3]{\mathcal T(#1,#2,#3)} 
\newcommand {\trd}[3]{\mathcal T(\s{#1},\s{#2},\s{#3})} 
\newcommand {\al}{\alpha}
\newcommand {\zsum}{\displaystyle\sum} 
\newcommand {\ot}{{t_{|T|}}}

\newcommand {\uu}{u_1}
\newcommand {\ud}{u_2}
\newcommand {\ci}{c_i}
\newcommand {\ciu}{c_{i+1}}
\newcommand {\cid}{c_{i+2}}

\newcommand {\Ot}{\overline t}
\newcommand {\Tt}{\tilde t}
\newcommand {\Os}{\overline s}
\newcommand {\Ts}{\tilde s}

\newcommand {\id}[1]{\,\mathbb{1}\s{#1}}

\newcommand {\s}[1]{\{#1\}}

\newcommand{\trtikz}[7]{

  \begin{tikzpicture}[baseline=(current bounding box.center)]

    \newcommand\ofhUn{0.23}; 
    \newcommand\ofhDeux{0.23};   
    \newcommand\cw{0.3cm}; 
    \newcommand\ch{0.6cm}; 

    \tikzstyle{cluster}=[draw,  inner  sep=1.5pt,  minimum  width=0.6cm,
    minimum  height  = 1cm,  text  centered,  thick, circle,  densely
    dashed, fill=gray!20]

    \tikzstyle{fix_node}=[inner   sep=0pt,minimum   height=0mm,
        minimum width = 0mm, text centered];

    \tikzstyle{move_node}=[fix_node,  ellipse,  minimum  height=0.4cm,
    minimum width=0.4cm, draw, densely dotted];

    \tikzstyle{fleche}=[->, >=latex]

    \ifthenelse{\equal{#1}{}}{
      \renewcommand\ofhUn{0};
    }{
      \node[fix_node] (n1) at (0,\ofhUn){#1};
      \begin{pgfonlayer}{background}
      \end{pgfonlayer}
    }

    \ifthenelse{\equal{#3}{}}{
      \node[fix_node] (n2) at (0,-\ofhUn){#2};
    }{
      \node[move_node] (n2) at (0,-\ofhUn){#2};
    }

    \begin{pgfonlayer}{background}
      \ifthenelse{\equal{#1}{}}{
        \node[cluster, fit=(n2)] (c1) at (0,0){};
      }{
        \node[cluster, fit=(n1) (n2)] (c1) at (0,0){};
      }
    \end{pgfonlayer}

    \ifthenelse{\equal{#4}{}}{
      \renewcommand\ofhDeux{0};
    }{
      \node[fix_node] (n4) at (#7,-\ofhDeux) {#4};
    }

    \ifthenelse{\equal{#6}{}}{
      \node[fix_node] (n3) at(#7,\ofhDeux){#5};
    }{
      \node[move_node] (n3) at(#7,\ofhDeux){#5};
    }

    \begin{pgfonlayer}{background}
      \ifthenelse{\equal{#4}{}}{
        \node[cluster, fit=(n3)] (c2) {};
      }{
        \node[cluster, fit=(n3) (n4)] (c2) {};
      }
    \end{pgfonlayer}
    
    \ifthenelse{\equal{#3}{}}{}{
    \draw[fleche] (n2) to[out=-20, in=-160] 
       (c2);
     }

    \ifthenelse{\equal{#6}{}}{}{
    \draw[fleche] (n3) to[out=160, in=20] 
       (c1);
     }

  \end{tikzpicture}
}

\usepackage[color=myyellow,linecolor=blue,bordercolor=blue]{todonotes}

\journal{Discrete Applied Mathematics}

\begin{document}

\begin{frontmatter}



\title{On  the   polyhedron  of  the   $K$-partitioning  problem  with
  representative variables}

\author[label1,label2]{Zacharie \textsc{Ales}}
\author[label1]{Arnaud \textsc{Knippel}\corref{cor1}}
\cortext[label2]{arnaud.knippel@insa-rouen.fr,\\\indent  +33 2 32  95 66
  25, \\\indent INSA de
 Rouen, Avenue de l'Universit\'e, 76801, Saint-Etienne du Rouvray}
\author[label2]{Alexandre \textsc{Pauchet}}

\address[label1]{LMI INSA Rouen (EA 3226)}
\address[label2]{LITIS INSA Rouen (EA 4051)}

\begin{abstract}

  The $K-$partitioning  problem consists of  partitioning the vertices
  of a  graph in $K$  sets so  as to minimize  a function of  the edge
  weights.  We introduce a linear mixed integer formulation with edge
  variables   and   representative   variables.    We   consider   the
  corresponding   polyhedron   and   show   which   inequalities   are
  facet-defining.   We   study  several  families   of  facet-defining
  inequalities  and  provide experimental  results  showing that  they
  improve significantly the linear relaxation of our formulation.



\end{abstract}


\end{frontmatter}


\section{Introduction}
\label{sec:introduction}

Graph partitioning refers  to partitioning the vertices of  a graph in
several  sets,  so  that a  given  function  of  the edge  weights  is
minimized (or maximized).   In many papers this function  is linear so
that  minimizing the weight  of the  edges in  the different  parts is
equivalent to maximizing  the total weight of the  multicut defined by
the  node  sets. For  this  reason different  names  are  used in  the
literature   and   the    most   frequent   are   graph   partitioning
problem~\cite{chopra1993partition,bichot2013graph,ferreira1998node}
and min-cut  problem~\cite{johnson1993min}.  When maximizing  the cut,
with   positive   weights,  the   problem   is   called  the   max-cut
problem~\cite{mahjoub1986polytope,garey1976some}  and is  known  to be
NP-complete.   The  problem is  sometimes  called clique  partitioning
problem            when             the            graph            is
complete~\cite{grotschel1990facets,oosten2001clique}, while Chopra and
Rao~\cite{chopra1993partition}  note  that  the  general case  can  be
solved by adding  edges to obtain a complete graph.  This is the point
of  view we  adopt here  for the  sake of  simplicity.  However  it is
possible to  derive specific valid inequalities for  sparse graphs, as
in~\cite{ferreira1998node}.

This  general   problem  has   many  applications  (see   for  example
\cite{bichot2013graph})  and many  variants, which  in most  cases are
NP-hard~\cite{garey1976some}.  The number of sets in a solution may be
specified as a  part of the problem definition or  not.  In this paper
we consider  the former case that we  call the \textit{K-partitioning}
problem, where  $K$ is the number of  sets.  



Our motivation  comes from a clustering problem  for analyzing dialogs
in  psychology \cite{alesmethodology}.  Dialogs  can be  encoded using
two-dimensional tables  (or series  of item-sets), along  which dialog
patterns,   representative   of    human   behaviors,   are   repeated
approximately.  Partitioning  a graph of dialog  patterns would enable
to group  similar instances and therefore  to characterize significant
behaviors.  For this  application, instances  commonly are  a complete
graph of 20 to 100 vertices to  partition in 6 to 10 sets. Whereas big
instances  require   approximate  solutions,  we   are  interested  in
improving  the exact methods  based on  branch and  bound (such  as in
\cite{fan2010linear,hager2013exact,bonami2012solution}),  by  studying
more  precisely the polyhedral  structure of  a linear  formulation in
order to provide better bounds.


A general linear integer formulation using edge variables was proposed
in  \cite{grotschel1990facets}, together  with  several facet-defining
inequalities.  We call this  formulation edge formulation or node-node
formulation, and it is often considered stronger than the node-cluster
formulation~\cite{bonami2012solution}, although this may depend on the
data  sets.  However  the edge  formulation doesn’t  allow to  fix the
number   of   sets   easily,   as   opposite   to   the   node-cluster
formulation. Some  authors use a formulation with  both edge variables
and node-cluster  variables. A formulation with  an exponential number
of  constraints has  been  proposed in~\cite{ferreira1996formulations}
for  a  variant of  partitioning  with  a bound  on  the  size of  the
clusters, applied  to sparse  graphs.  More compact  formulations have
been   proposed  based   on   the  linearization   of  the   quadratic
formulation~\cite{kaibel2011orbitopal,fan2010linear}.
In~\cite{kaibel2011orbitopal}  the weights of  the edges  are positive
and the total edge weight is
minimized, while  in~\cite{fan2010linear} it is
maximized.  When  considering  the  triangle  inequalities  from  both
formulations~\cite{chopra1993partition}, graph  with arbitrary weights
on the edges can be partitioned.

All these formulations have a common  drawback: they contain a lot of
symmetry and this can considerably slow down methods based on branch and bound or
branch and  cut.  A way  to deal with the  symmetry is to
work on the branching strategy.  Kaibel et
al~\cite{kaibel2011orbitopal}  have proposed  a  general tool,  called
orbitopal  fixing for  that  purpose.   Another way  is  to break  the
symmetry directly in the  formulation.  This approach has already been
used  in~\cite{campelo2008asymmetric} for  a vertex  coloring problem.
More recently a similar idea has been applied to break the symmetry in
the node-cluster formulation~\cite{bonami2012solution}.  In this paper
we   propose    a   formulation   based   on    the   edge   variables
of~\cite{grotschel1990facets,labbe2010size}  and  additional variables
that we call  representative variables. This allows not  only to break
the symmetry, but also to fix the number $K$ of sets.


We present  in Section~\ref{sec:formulation} our  formulation.  We study in  Section~\ref{sec:dimension} the
dimension of  the polyhedron  $P_{n,K}$ associated to  our formulation.
We  characterize   in  Section~\ref{sec:trivial} all  the   trivial  facet-defining
inequalities. In Sections~\ref{sec:chorded},~\ref{sec:2part} and~\ref{sec:general}, we respectively study the so-called $2-$chorded
cycle  inequalities,  the  2-partition  inequalities and  the  general
clique inequalities and we determine  cases where  they define
facets  of  $P_{n,K}$.   In   Section~\ref{sec:strength}  we  strengthen  the  triangle
inequalities from  our formulation  that do not  define facets  and we
show that most of the  time the strengthened triangle inequalities are
facet-defining.  In Section~\ref{sec:paw} we study a new family of inequalities
called \textit{paw inequalities} and identify when
they  are  facet-defining.  In  the  last  section  we illustrate  the
improvement on the linear relaxation  value of our formulation for the
facet-defining  inequalities of  the previous  sections,  for complete
graphs with different kind of weights.

\section{A mixed integer linear formulation}
\label{sec:formulation}

Let $V=\{1,\hdots, n\}$ be a set of indexed vertices and $G=(V,E)$ the
complete  graph induced  by  $V$. A  $K\mbox{-partition}$  $\pi$ is  a
collection of  $K$ non-empty subsets  $C_1, C_2, \hdots,  C_K$, called
clusters, such  that $ \forall  i\neq j,$ $C_i\cap  C_j=\emptyset$ and
$\bigcup_{i=1}^K C_i = V$.

To  each   $K$-partition  $\pi$,  we  associate   a  characteristic  vector
$x^\pi\in\{0,1\}^{|E|+|V|}$ such that:
\begin{itemize}
\item for each edge $uv\in E$, $x^\pi_{u,v}$ (equivalent to
  $x^\pi_{v,u}$) is equal to $1$ if $u,v\in
  C_i$ for some
  $i$ in $\{1, 2, \hdots, K\}$ and $0$ otherwise;
\item  for each  vertex $u\in  V$, $x^\pi_u=1$  if $u$  is the
 vertex with the smallest index of its cluster (in that case $u$ is said to
 be the \textit{representative} of its cluster) and $0$ otherwise.
\end{itemize}

An edge $uv\in E$  is said to be \textit{activated} for a given
partition  $\pi$ if  $x^\pi_{u,v}=1$. In  this  context, $u$  is said  to
\textit{be linked} to $v$ and vice versa.  A vertex $i$ is said to be
\textit{lower} than another vertex $j$ (noted $i<j$) if index $i$ is lower than
index $j$.

Let  $d_{i,j}$ denote the  cost of  edge $ij\in  E$.  We  consider the
following formulation for the $K-$partitioning problem:

\begin{center}                                                              
  $(P_1)\left\{ \begin{array}{lll}\min
      \displaystyle{\sum_{ij\in E}d_{i,j}x_{i,j}}\\
      x_{i,k}+x_{j,k}-x_{i,j}\leq  1  &  \forall i,j,k\in  V,\,  i\neq
      k,\,j\neq k,\, i<j & \tagarray\\
      x_j +x_{i,j} \leq 1& \forall i,j\in V,\, i<j&
      \tagarray\label{eq:lrep}\\
      x_j + \displaystyle\sum_{i=1}^{j-1} x_{i,j}\geq 1 & \forall j\in
      V& \tagarray\label{eq:urep}\\
      \displaystyle\sum_{i=1}^n x_i = K & & \tagarray\label{eq:nbClusters}\\
      x_{i,j}\in\{ 0,1\} & ij\in E\\
      x_i\in[0,1] & i\in V\end{array}\right.$
\end{center} 

Constraints (1), called \textit{triangle inequalities}, ensure that if
two  incident edges  $ij$ and  $jk$ are  activated then  $ik$  is also
activated.    Note  that   their   are  $\frac{n(n-1)(n-2)}{2}$   such
constraints (three for each triangle ${a,b,c}$ of $G=(V,E)$). Constraints $(2)$,
called \textit{upper  representative inequalities}, ensure  that every
cluster  contains  no  more  than  one representative.  If  $j$  is  a
representative then it is not linked  to any lower vertex $i$.  If $j$
is linked  to such  a lower  vertex then it  is not  a representative.
Constraints $(3)$,  called \textit{lower representative inequalities},
guarantee that a cluster contains at least one representative. Indeed,
on the one hand if $j$ is not a representative then it is linked to at
least one lower vertex, on the other  hand if $j$ is not linked to any
of these vertices then it is a representative. Finally, constraint (4)
ensures that the number of clusters is equal to $K$.

Note  that  in  the  above  mixed linear  program  the  representative
variables can be relaxed as fixing all edge
variables to $0$ or $1$  forces the representative variables to be in
$\s{0,1}$. Hence, we only have $|E|$ binary variables.

As the polyhedron  associated to the above formulation  is  not
full-dimensional,  we fix $x_1$  to $1$  (since vertex  $1$ is  always a
representative)  and  substitute $x_2$  by  $1-x_{1,2}$  and $x_3$  by
$K-2+x_{1,2}-\sum_{i=4}^n  x_i$.  Equation  $(4)$ can  now  be removed
since the number of clusters in a solution is ensured to be $K$ by the
expression  substituted to  $x_3$.   These modifications  lead to  the
following linear mixed integer program:\\

            \[(P_2)\left\{ \begin{array}{ll@{\hspace{0.1cm}}l}\min
                \displaystyle{\sum_{ij\in E}w_{i,j}x_{i,j}}\\
                x_{i,j}+x_{j,k}-x_{i,k}\leq 1  & \forall i,j,k\in V,\,
                i\neq
                j,\,i\neq k,\, j<k & (1)\\

                x_j +x_{i,j} \leq 1& \forall i\in V,j\in V-\{1,2,3\},\, i<j&
                (2)\\

                \displaystyle\sum_{j=4}^n x_j - x_{1,2} - x_{i,3} \geq K-3 & \forall
                i\in \{1,2\} & (2')\\

                x_j  + \displaystyle\sum_{i=1}^{j-1}  x_{i,j}\geq  1 &
                \forall j\in
                V& (3)\\

                \displaystyle\sum_{i=4}^n x_i - x_{1,2} - x_{1,3} - x_{2,3} \leq K-3 &
                &(3')\\

                \displaystyle\sum_{i=4}^n x_i - x_{1,2} \geq K-3 &
                & 
                \refstepcounter{equation}(\theequation)
\\

                \displaystyle\sum_{i=4}^n  x_i  -  x_{1,2}\leq  K-2  &
                & \refstepcounter{equation}(\theequation) \\

                x_{i,j}\in\{ 0,1\} & ij\in E\\
                x_i\in[0,1] & i\in V\end{array}\right..\]

As  a result, the
characteristic  vector of  a partition  $\pi$ no  longer  contains the
components $x_1$, $x_2$ and  $x_3$.  The vector $x^{\pi}$ now contains
the $n-3$  remaining representative  components followed by  the $|E|$
edges components:
\[(x^\pi)^T=(x_4,  \hdots,  x_n,  x_{1,2}, \hdots,  x_{1,n},  x_{2,3},
\hdots, x_{n-1,n}).\]

However, for a given vector $\al\in\mathbb R^{|E|+|V|-3}$ the three
coefficients related  to $x_1$, $x_2$ and $x_3$  ($\al_1$, $\al_2$ and
$\al_3$) may appear
in subsequent proofs to assist the understanding. In that case, they are  equal to zero.
\bigskip

Let $P_{n,K}$ be  the convex hull of all integer
points which are feasible for $(P_2)$:
\[ P_{n,K} = conv\s{x\in\s{0,1}^{|E|+|V|-3} |\, x \mbox{ satisfies } (1),
  (2), (2'), (3), (3'), (5), (6)}\]

To simplify the notations, a  singleton $\s{s}$ may be denoted by $s$.
Likewise for a given vector $\al\in\mathbb R^{|E|+|V|-3}$ and a subset
$E_1$ of  $E$, the term  $\al(E_1)$ is used  to denote the sum  of the
$\al$ components in $E_1$  ($\sum_{e\in E_1} \al_{e}$). Finally, if
we  consider two  subsets of  $V$,  $V_1$ and  $V_2$, the  sum of  the
$\alpha$   inter-set   components   $\sum_{i\in  V_1}\sum_{j\in   V_2}
\al_{i,j}$  and   the  sum   of  the  $\alpha$   intra-set  components
$\sum_{i,j\in  V_1,  i<j}   \al_{i,j}$  are  respectively  denoted  as
$\al(V_1, V_2)$ and $\al(V_1)$.







The truth function $\mathbb{1}$ of a boolean expression $e$ (noted $\id{e}$) is
equal to $1$ if $e$ is true and $0$ otherwise.

Given two disjoint clusters $C_1$,  $C_2$ and a set of vertices $R\subset
C_1\cup  C_2$ we  define  the following  transformation: $\mathcal  T:
\{C_1,C_2,R\}\mapsto \{C_1', C_2'\}$,  with $ C_1'=(C_1\backslash R)\cup
(R\backslash C_1)$ and $ C_2'=(C_2\backslash R)\cup (R\backslash C_2)<$.  The
corresponding        transformation       is        presented       in
Figure~\ref{fig:transformationDef}.  This  operator will be  used in the
following   to  highlight  relations   between  the   coefficients  of
hyperplanes of $\mathbb R^{|E|+|V|-3}$.

    \begin{figure}[h]
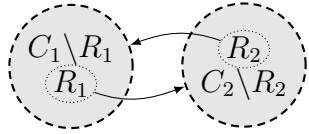

      \centering
      \trtikz{$C_1\backslash  R_1$} {$R_1$} {1}  {$C_2\backslash R_2$}
      {$R_2$} {1} {2.3}

    \caption{Representation of  $\mathcal  T(C_1, C_2,  R)$
      with $R_1=R\cap C_1$ and $R_2=R\cap C_2$.}
      \label{fig:transformationDef}
  \end{figure}

\section{Dimension of $P_{n,K}$}
\label{sec:dimension}

Let  $\pi=\{C_1, C_2,$  $\hdots, C_K\}$  be a  $K-$partition.  For all
$i\in\s{1,\hdots,K}$,  the representative vertices  of clusters  $C_i$ is
referred  to as  $r_i$.  The  second lowest  vertex of  this  cluster is
denoted by $r'_i=$ $\min \s{j\in C_i\backslash\s{r_i}}$.

To   prove  that  $P_{n,K}$   is  full-dimensional   if  and   only  if
$K\in\{3,4,\hdots, $ $n-2\}$, we first assume that it is included in a
hyperplane $H=\{  x\in\mathbb R^{|E|+|V|-3}$ $| \;\al^T  x = \al_0\}$.
We then prove that
every $\alpha$ coefficient  is equal to zero. To this  end, we rely on
four lemmas (summed up in table~\ref{tab:lemmas}).

\begin{table}[h]

  \centering
  \begin{tabular}{cl@{\hspace{0.2cm}}c@{}c}\toprule
    \multirow{2}{*}{\textbf{Lemma}} & \multirow{2}{*}{\textbf{Conditions}} & \textbf{Valid} &
    \multirow{2}{*}{\textbf{Result}}\\
    & & \textbf{transformations} & \\\midrule

    \multirow{3}{*}{\ref{lemma1}} & $c_1\in C_1\backslash\s{r_1},$ & \multirow{3}{*}{
\begingroup%
  \makeatletter%
  \providecommand\color[2][]{%
    \errmessage{(Inkscape) Color is used for the text in Inkscape, but the package 'color.sty' is not loaded}%
    \renewcommand\color[2][]{}%
  }%
  \providecommand\transparent[1]{%
    \errmessage{(Inkscape) Transparency is used (non-zero) for the text in Inkscape, but the package 'transparent.sty' is not loaded}%
    \renewcommand\transparent[1]{}%
  }%
  \providecommand\rotatebox[2]{#2}%
  \ifx\svgwidth\undefined%
    \setlength{\unitlength}{70.42666016bp}%
    \ifx\svgscale\undefined%
      \relax%
    \else%
      \setlength{\unitlength}{\unitlength * \real{\svgscale}}%
    \fi%
  \else%
    \setlength{\unitlength}{\svgwidth}%
  \fi%
  \global\let\svgwidth\undefined%
  \global\let\svgscale\undefined%
  \makeatother%
  \begin{picture}(1,0.49879699)%
    \put(0,0){\includegraphics[width=\unitlength]{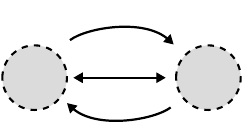}}%
    \put(0.85633326,0.14062574){\color[rgb]{0,0,0}\makebox(0,0)[b]{\smash{$C_{2}$}}}%
    \put(0.14370002,0.14062574){\color[rgb]{0,0,0}\makebox(0,0)[b]{\smash{$C_{1}$}}}%
    \put(0.4888601,0.23646528){\color[rgb]{0,0,0}\makebox(0,0)[b]{\smash{$c_1,c_2$}}}%
    \put(0.49533535,0.4470144){\color[rgb]{0,0,0}\makebox(0,0)[b]{\smash{$c_1$}}}%
    \put(0.49533535,0.06688228){\color[rgb]{0,0,0}\makebox(0,0)[b]{\smash{$c_2$}}}%
  \end{picture}%
\endgroup%
}
      & \multirow{3}{*}{$\al_{c_1, c_2} = 0$}
      \\
      & $c_2\in C_2\backslash\s{r_2}$ & & 
      \\
      & & & 
      \\\midrule

      \multirow{3}{*}{\ref{lemma2} }             &              $r_1<r_2$             &
      \multirow{3}{*} {
\begingroup%
  \makeatletter%
  \providecommand\color[2][]{%
    \errmessage{(Inkscape) Color is used for the text in Inkscape, but the package 'color.sty' is not loaded}%
    \renewcommand\color[2][]{}%
  }%
  \providecommand\transparent[1]{%
    \errmessage{(Inkscape) Transparency is used (non-zero) for the text in Inkscape, but the package 'transparent.sty' is not loaded}%
    \renewcommand\transparent[1]{}%
  }%
  \providecommand\rotatebox[2]{#2}%
  \ifx\svgwidth\undefined%
    \setlength{\unitlength}{70.42666016bp}%
    \ifx\svgscale\undefined%
      \relax%
    \else%
      \setlength{\unitlength}{\unitlength * \real{\svgscale}}%
    \fi%
  \else%
    \setlength{\unitlength}{\svgwidth}%
  \fi%
  \global\let\svgwidth\undefined%
  \global\let\svgscale\undefined%
  \makeatother%
  \begin{picture}(1,0.49879699)%
    \put(0,0){\includegraphics[width=\unitlength]{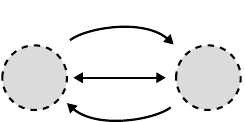}}%
    \put(0.85633326,0.14062574){\color[rgb]{0,0,0}\makebox(0,0)[b]{\smash{$C_{2}$}}}%
    \put(0.14370002,0.14062574){\color[rgb]{0,0,0}\makebox(0,0)[b]{\smash{$C_{1}$}}}%
    \put(0.4888601,0.23646528){\color[rgb]{0,0,0}\makebox(0,0)[b]{\smash{$r_1,c_2$}}}%
    \put(0.49533535,0.4470144){\color[rgb]{0,0,0}\makebox(0,0)[b]{\smash{$r_1$}}}%
    \put(0.49533535,0.06688228){\color[rgb]{0,0,0}\makebox(0,0)[b]{\smash{$c_2$}}}%
  \end{picture}%
\endgroup%
} &
      \multirow{3}{*}{$2\al_{r_1,c_2}+\al_{r'_1}=\al_{\min(r'_1,c_2)}$ }
      \\
      & $c_2\in C_2\backslash\s{r_2}$ 
      \\
      \\\midrule

      \multirow{3}{*}{\ref{lemma4} }             &              $r_1<r_2$             &
      \multirow{3}{*} {
\begingroup%
  \makeatletter%
  \providecommand\color[2][]{%
    \errmessage{(Inkscape) Color is used for the text in Inkscape, but the package 'color.sty' is not loaded}%
    \renewcommand\color[2][]{}%
  }%
  \providecommand\transparent[1]{%
    \errmessage{(Inkscape) Transparency is used (non-zero) for the text in Inkscape, but the package 'transparent.sty' is not loaded}%
    \renewcommand\transparent[1]{}%
  }%
  \providecommand\rotatebox[2]{#2}%
  \ifx\svgwidth\undefined%
    \setlength{\unitlength}{70.42666016bp}%
    \ifx\svgscale\undefined%
      \relax%
    \else%
      \setlength{\unitlength}{\unitlength * \real{\svgscale}}%
    \fi%
  \else%
    \setlength{\unitlength}{\svgwidth}%
  \fi%
  \global\let\svgwidth\undefined%
  \global\let\svgscale\undefined%
  \makeatother%
  \begin{picture}(1,0.49879699)%
    \put(0,0){\includegraphics[width=\unitlength]{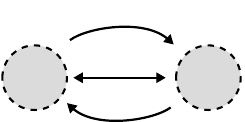}}%
    \put(0.85633326,0.14062574){\color[rgb]{0,0,0}\makebox(0,0)[b]{\smash{$C_{2}$}}}%
    \put(0.14370002,0.14062574){\color[rgb]{0,0,0}\makebox(0,0)[b]{\smash{$C_{1}$}}}%
    \put(0.4888601,0.23646528){\color[rgb]{0,0,0}\makebox(0,0)[b]{\smash{$r_1,r_2$}}}%
    \put(0.49533535,0.4470144){\color[rgb]{0,0,0}\makebox(0,0)[b]{\smash{$r_1$}}}%
    \put(0.49533535,0.06688228){\color[rgb]{0,0,0}\makebox(0,0)[b]{\smash{$r_2$}}}%
  \end{picture}%
\endgroup%
} &
      \multirow{3}{*}{$2\al_{r_1,r_2} + \al_{r'_1}+\al_{r'_2}=2\al_{r_2}$ }
      \\
      & $r_2 < r'_1$ 
      \\
      & $|C_2|\geq 2$\\\midrule

      \multirow{3}{*}{\ref{lemma5} }             &           $i\in C_1\backslash\s{r_1}$              &
      \multirow{3}{*} {
\begingroup%
  \makeatletter%
  \providecommand\color[2][]{%
    \errmessage{(Inkscape) Color is used for the text in Inkscape, but the package 'color.sty' is not loaded}%
    \renewcommand\color[2][]{}%
  }%
  \providecommand\transparent[1]{%
    \errmessage{(Inkscape) Transparency is used (non-zero) for the text in Inkscape, but the package 'transparent.sty' is not loaded}%
    \renewcommand\transparent[1]{}%
  }%
  \providecommand\rotatebox[2]{#2}%
  \ifx\svgwidth\undefined%
    \setlength{\unitlength}{77.30767211bp}%
    \ifx\svgscale\undefined%
      \relax%
    \else%
      \setlength{\unitlength}{\unitlength * \real{\svgscale}}%
    \fi%
  \else%
    \setlength{\unitlength}{\svgwidth}%
  \fi%
  \global\let\svgwidth\undefined%
  \global\let\svgscale\undefined%
  \makeatother%
  \begin{picture}(1,0.55561893)%
    \put(0,0){\includegraphics[width=\unitlength]{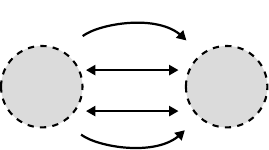}}%
    \put(0.84808949,0.19569599){\color[rgb]{0,0,0}\makebox(0,0)[b]{\smash{$C_{2}$}}}%
    \put(0.15749337,0.19569599){\color[rgb]{0,0,0}\makebox(0,0)[b]{\smash{$C_{1}$}}}%
    \put(0.49262792,0.34509459){\color[rgb]{0,0,0}\makebox(0,0)[b]{\smash{$r_1,r_2$}}}%
    \put(0.49852682,0.50844541){\color[rgb]{0,0,0}\makebox(0,0)[b]{\smash{$r_1$}}}%
    \put(0.49852682,0.02540317){\color[rgb]{0,0,0}\makebox(0,0)[b]{\smash{$i$}}}%
    \put(0.49262792,0.17393426){\color[rgb]{0,0,0}\makebox(0,0)[b]{\smash{$i,r_2$}}}%
  \end{picture}%
\endgroup%
} &
      \multirow{3}{*}{$\al_{r_1,r_2}=\al_{r_2,i}$ }
      \\
      & $i<r_2$
      \\
      &
      \\

    \bottomrule
  \end{tabular}
  \caption{Summary   of   the   four   lemmas.   Given   a   $K-$partition
    $\pi=\s{C_1,\hdots,C_K}$ whose vector $x^\pi$ is included in a hyperplane $H=\{ x\in\mathbb R^{|E|+|V|-3}$ $| \;\al^T x =
\al_0\}$, each arrow of the third column
    corresponds to a transformation on $C_1$ and $C_2$ which has to give a
     $K-$partition included in $H$.}
  \label{tab:lemmas}
\end{table}

\begin{lemma}
\label{lemma1}Consider the four following $K$-partitions :
  \begin{enumerate}[(i)]
  \item $\pi =  \{C_1, C_2, C_3, \hdots, C_K\}$  with $c_1\in C_1\backslash\s{r_1}$ and
    $c_2\in C_2\backslash\s{r_2}$;
  \item  $\pi^1 =  \{C^{(1)}_1,  C^{(1)}_2, C_3,  \hdots, C_K\}$  with
    $\{C^{(1)}_1, C^{(1)}_2\} = \mathcal T(C_1, C_2, \s{c_1})$;
  \item  $\pi^2 =  \{C^{(2)}_1,  C^{(2)}_2, C_3,  \hdots, C_K\}$  with
    $\{C^{(2)}_1, C^{(2)}_2\} =\mathcal T(C_1, C_2, \s{c_2})$;
  \item  $\pi^3 =  \{C^{(3)}_1,  C^{(3)}_2, C_3,  \hdots, C_K\}$  with
    $\{C^{(3)}_1, C^{(3)}_2\} =\mathcal T(C_1, C_2, \{c_1, c_2\})$.
  \end{enumerate}
  If $x^{\pi}$, $x^{\pi^1}$, $x^{\pi^2}$, $x^{\pi^3}$ satisfy
  $\al^Tx=\al_0$  then $\al_{c_1,c_2} = 0$.
\end{lemma}


\begin{figure}[h!]
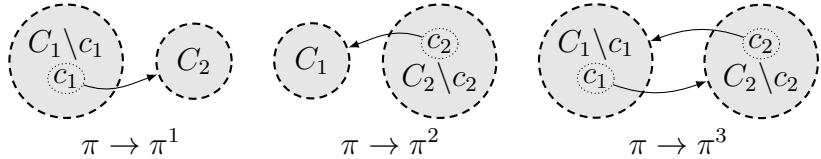


  \centering 
\begin{tabular}{c@{\hspace{0.4cm}}c@{\hspace{0.4cm}}c}
            \trtikz{$C_1\backslash  c_1$}{$c_1$}{1}{}{$C_2$}{}{1.7}  &
            \trtikz{}{$C_1$}{}{$C_2\backslash  c_2$}{$c_2$}{1}{1.7}  &
            \trtikz{$C_1\backslash       c_1$}{$c_1$}{1}{$C_2\backslash
              c_2$}{$c_2$}{1}{2.2}\\\rule[-7pt]{0pt}{20pt}
            $\pi\rightarrow \pi^1$ & 
            $\pi\rightarrow \pi^2$ &
            $\pi\rightarrow \pi^3$
  \end{tabular}
  \caption{Representation of the transformations which lead from $\pi$ to $\pi^1$, $\pi^2$ and $\pi^3$.}
\label{fig:transformations}
\end{figure}

\begin{proof}
Let $m_2$ denote $min(r_2, c_1)$.  Using the fact that $\al^T
x^{\pi}$ and $\al^Tx^{\pi^1}$ are both equal to $\al_0$, we obtain

\begin{multline}
  \zsum_{i=1}^K(\al(C_i) +\al_{r_i}) = \zsum_{i=3}^K (\al(C_i)
  +\al_{r_i})\\+ \al(C_1\backslash\s{c_1}) + \al(C_2\cup \{c_1\})
  + \al_{r_1} + \al_{m_2}.
\end{multline}

This can be simplified and reformulated as
\begin{equation}
  \label{eq:l1e1}
  \al(\s{c_1}, C_1) - \al(\s{c_1}, C_2\backslash\s{c_2}) =\\ \al_{c_1,c_2} + \al_{m_2} - \al_{r_2}.
\end{equation}

Let $m_1$ correspond to $min(r_1, c_2)$. Similarly we obtain from $\pi^2$

\begin{equation}
  \label{eq:l1e2}
 \al(\s{c_2}, C_2) - \al(\s{c_2},C_1\backslash\s{c_1})= \al_{c_1,c_2} + \al_{m_1} - \al_{r_1}.
\end{equation}

Finally, $\pi^3$ leads to
\begin{multline}
  \label{eq:l1e3}
  \al(\s{c_1}, C_1) +\al(\s{c_2}, C_2) + \al_{r_1} + \al_{r_2}.
  = \\\al(\s{c_2}, C_1\backslash\s{c_1}) + \al(\s{c_1}, C_2\backslash\s{c_2}) + \al_{m_1} + \al_{m_2} 
\end{multline}

From (\ref{eq:l1e1}), (\ref{eq:l1e2}) and (\ref{eq:l1e3}) we obtain $\al_{c_1 c_2}
= 0$.
\end{proof}

\begin{lemma}
\label{lemma2}Consider the four following $K$-partitions :
  \begin{enumerate}[(i)]
  \item $\pi = \{C_1, C_2, C_3, \hdots, C_K\}$ with $r_1<r_2$, 
    $c_2\in C_2\backslash\s{r_2}$ and $|C_1|\geq 2$;
  \item  $\pi^1 =  \{C^{(1)}_1,  C^{(1)}_2, C_3,  \hdots, C_K\}$  with
    $\{C^{(1)}_1, C^{(1)}_2\} = \mathcal T(C_1, C_2, \s{r_1})$;
  \item  $\pi^2 =  \{C^{(2)}_1,  C^{(2)}_2, C_3,  \hdots, C_K\}$  with
    $\{C^{(2)}_1, C^{(2)}_2\} =\mathcal T(C_1, C_2, \s{c_2})$;
  \item  $\pi^3 =  \{C^{(3)}_1,  C^{(3)}_2, C_3,  \hdots, C_K\}$  with
    $\{C^{(3)}_1, C^{(3)}_2\} =\mathcal T(C_1, C_2, \{r_1, c_2\})$.
  \end{enumerate}
  If $x^{\pi}$, $x^{\pi^1}$, $x^{\pi^2}$, $x^{\pi^3}$ satisfy
  $\al^Tx=\al_0$  then $2\al_{r_1, c_2}+ \al_{r'_1}=\al_{\min(r'_1,c_2)}$.
\end{lemma}

\begin{proof}

The transformation which leads from $\pi$ to $\pi^1$ is represented in
Figure~\ref{fig:l2f1}. Since $\al^Tx^\pi=\al^Tx^{\pi^1}$ we deduce:
  \begin{equation}
\al(\s{r_1}, C_1) +\al_{r_2}=\al(\s{r_1}, C_2) +\al_{r'_1}\label{eq:l2e1}.
\end{equation}

  In $\pi^1$, $r_2$ is  not a representative vertex since
  $r_1$ - which is lower - is in its cluster.  As a result $\al_{r_2}$
  only appears in the left part of the equation. At the opposite,
  $r'_1$,  is not  a  representative in  $\pi$,  but becomes  one 
  when $r_1$ is moved to $C_2$.  Therefore, $\al_{r'_1}$
  appears in the right part of the equation $\al^T x^\pi = \al^Tx^{\pi^1}$.
  \begin{figure}[h]
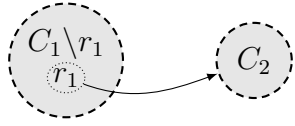

    \centering
    \trtikz{$C_1\backslash r_1$}{$r_1$}{1}{}{$C_2$}{}{2.5}
    \caption{Representation  of   the  transformation  which
      leads from $\pi$ to $\pi^1$ when $c_1=r_1$.}
    \label{fig:l2f1}
  \end{figure}

We deduce from  $\al^Tx^\pi=\al^Tx^{\pi^2}$:
\begin{equation}
  \label{eq:l2e2}
  \al(\s{c_2}, C_2) = \al(\s{c_2}, C_1).
\end{equation}

Furthermore, the fact that $\al^Tx^\pi$ is equal to $\al^Tx^{\pi^3}$ gives

\begin{multline}
  \al(\s{r_1},  C_1) + \al(\s{c_2},  C_2) +  \al_{r_2} =\\
  \al(\s{r_1},     C_2\backslash\s{c_2})     +     \al(\s{c_2},
  C_1\backslash r_1) + \al_{\min(r'_1,c_2)}.
\end{multline}

Finally, this result can be simplified to $\al_{r_1,c_2} + \al_{r_2}= \al_{\min(r'_1,c_2)}$ with
equations (\ref{eq:l2e1}) and (\ref{eq:l2e2}).
\end{proof}

\begin{lemma}
\label{lemma4}Consider the four following $K$-partitions :
  \begin{enumerate}[(i)]
  \item  $\pi = \{C_1,  C_2, C_3,  \hdots, C_K\}$  with $r_1<r_2<r'_1$
    and $|C_2|\geq 2$;
  \item  $\pi^1 =  \{C^{(1)}_1,  C^{(1)}_2, C_3,  \hdots, C_K\}$  with
    $\{C^{(1)}_1, C^{(1)}_2\} = \mathcal T(C_1, C_2, \s{r_1})$;
  \item  $\pi^2 =  \{C^{(2)}_1,  C^{(2)}_2, C_3,  \hdots, C_K\}$  with
    $\{C^{(2)}_1, C^{(2)}_2\} =\mathcal T(C_1, C_2, \s{r_2})$;
  \item  $\pi^3 =  \{C^{(3)}_1,  C^{(3)}_2, C_3,  \hdots, C_K\}$  with
    $\{C^{(3)}_1, C^{(3)}_2\} =\mathcal T(C_1, C_2, \{r_1, r_2\})$.
  \end{enumerate}
  If $x^{\pi}$, $x^{\pi^1}$, $x^{\pi^2}$, $x^{\pi^3}$ satisfy
  $\al^Tx=\al_0$  then $2\al_{r_1,        r_2}        +       \al_{r'_1}        +
  \al_{r'_2}=2\al_{r_2}$.
\end{lemma}

\begin{proof}
  The fact that $\al^Tx^\pi$ is equal to $\al^Tx^{\pi^1}$ leads again to
  equation~(\ref{eq:l2e1}).

  From $\al^Tx^\pi=\al^Tx^{\pi^2}$ we deduce:
\begin{equation}
  \label{eq:l4e1}
  \al(\s{r_2}, C_2) +r_2= \al(\s{r_2}, C_1)+r'_2.
\end{equation}

The vertex $r_2$ is no more a representative in $\pi^1$ since it is
in the same cluster than $r_1$. $r'_2$ becomes a
representative since $r_2$ is not in its cluster anymore. 

Finally,  $\al^Tx^\pi=\al^Tx^{\pi^3}$ gives
\begin{multline}
\al(\s{r_1},      C_1)     +\al(\s{r_2},      C_2)     =
\al(\s{r_1}, C_2\backslash \s{r_2}) +
\al(\s{r_2}, C_1\backslash \s{r_1}).
\end{multline}

This  last  relation  can  be  simplified  to  $2\al_{r_1,  r_2}  +
\al_{r'_1} + \al_{r'_2}=2\al_{r_2}$ via equations~(\ref{eq:l2e1}) and~(\ref{eq:l4e1}).
\end{proof}

\begin{lemma}
\label{lemma5}Consider the five following $K$-partitions :
  \begin{enumerate}[(i)]
  \item   $\pi  =  \{C_1,   C_2,  C_3,   \hdots,  C_K\}$   with  $i\in
    C_1\backslash\s{r_1}$, $i<r_2$;
  \item  $\pi^1 =  \{C^{(1)}_1,  C^{(1)}_2, C_3,  \hdots, C_K\}$  with
    $\{C^{(1)}_1, C^{(1)}_2\} = \mathcal T(C_1, C_2, \s{r_1})$;
  \item  $\pi^2 =  \{C^{(2)}_1,  C^{(2)}_2, C_3,  \hdots, C_K\}$  with
    $\{C^{(2)}_1, C^{(2)}_2\} = \mathcal T(C_1, C_2, \s{i})$;
   \item  $\pi^3  = \{C^{(3)}_1,  C^{(3)}_2,  C_3,  \hdots, C_K\}$  with
    $\{C^{(3)}_1, C^{(3)}_2\} =\mathcal T(C_1, C_2, \{r_1,
    $ $r_2\})$;
  \item  $\pi^4  = \{C^{(4)}_1,  C^{(4)}_2,  C_3,  \hdots, C_K\}$  with
    $\{C^{(4)}_1, C^{(4)}_2\} =\mathcal T(C_1, C_2, \{i, r_2\})$.
  \end{enumerate}
  If $x^{\pi}$, $x^{\pi^1}$, $x^{\pi^2}$, $x^{\pi^3}$, $x^{\pi^4}$ satisfy
  $\al^Tx=\al_0$  then $\al_{r_1,  r_2}   =  \al_{r_2,  i}$.
\end{lemma}



\begin{figure}[h]
  \begin{center}
    \begin{tabular}{c*{1}{@{\hspace{1.2cm}}c}}
      \trtikz{$C_1\backslash r_1$}{$r_1$}{1}{}{$C_2$}{}{2.3} &
      \trtikz{$C_1\backslash i$}{$i$}{1}{}{$C_2$}{}{1.7} 
      \\\rule[-7pt]{0pt}{18pt}

      $\tr{C_1}{C_2}{r_1}$       &       $\tr{C_1}{C_2}{i}$      
   \\\rule[-7pt]{0pt}{18pt}
   \\

\trtikz{$C_1\backslash      r_1$}{$r_1$}{1}{$C_2\backslash
        r_2$}{$r_2$}{1}{3} & 
      \trtikz{$C_1\backslash                 i$}{$i$}{1}{$C_2\backslash
        r_2$}{$r_2$}{1}{2.4} \\\rule[-7pt]{0pt}{18pt}

      $\tr{C_1}{C_2}{\s{r_1,r_2}}$ & $\tr{C_1}{C_2}{\{i,r_2\}}$ 
    \end{tabular}
  \caption{Representation  of  the  transformations which  lead  from
    $\pi$ to $\pi^1_{r_1}$, $\pi^1_{i}$, $\pi^3$ and $\pi^4$.}
\label{fig:le5}
  \end{center}
\end{figure}

\begin{proof}
  As shown in the  previous proofs, we respectively obtain from
  $\al^Tx^\pi=\al^Tx^{\pi^1_{r_1}}$                    and
  $\al^Tx^\pi=\al^Tx^{\pi^1_i}$       equations~(\ref{eq:l2e1})
  and~(\ref{eq:l1e1}) (with $ c_2 = r_2$ in this last case).

  From $\al^Tx^\pi=\al^Tx^{\pi^3}$ (represented Figure~\ref{fig:le5}) we deduce:
  \begin{multline}
\al(\s{r_1},    C_1)+\al(\s{r_2},    C_2)+\al_{r_2}=    \al(\s{r_1},
C_2\backslash \s{r_2})\\+\al(\s{r_2}, C_1\backslash\s{r_1})+\al_{r'_1}.\label{eq:l5e1}
\end{multline}

The set $C_1^{(3)}$ is equal to $\{C_1\backslash c_1\}\cup \{r_2\}$.  Since
$r'_1$ is lower than $r_2$, it is $C_1^{(3)}$ representative.  

The equality $\al^Tx^\pi=\al^Tx^{\pi^4}$ shows that
\begin{multline}
  \label{eq:l5e2}
  \al(\s{i},    C_1)+\al(\s{r_2},    C_2)+\al_{r_2}=    \al(\s{i},
C_2\backslash\s{r_2})\\+\al(r_2, C_1\backslash\s{i})+\al_i.
\end{multline}

From equations~(\ref{eq:l2e1}) and~(\ref{eq:l5e1}) we get:
\begin{equation}
  \al(\s{r_2}, C_2)+\al_{r_1, r_2}=\al(\s{r_2}, C_1\backslash\s{r_1}),
\end{equation}

and from equations~(\ref{eq:l1e1}) and~(\ref{eq:l5e2}) we obtain:
\begin{equation}
  \al(\s{r_2}, C_2)+\al_{i, r_2}=\al(\s{r_2}, C_1\backslash\s{i}).
\end{equation}

Finally, the two last equations yield
the expected result.

\end{proof}






\begin{theorem}
Depending on $K$, the dimension of $P_{n,K}$ is:
  \begin{enumerate}[(i)]
  \item $dim(P_{n,2})=|E|
    +n-4$;
  \item $dim(P_{n,K})= |E|+n-3$,\; for $K\in\{3, 4, \hdots, n-2\}\;$(\textit{i.e.}: it is full
    dimensional);
  \item $dim(P_{n,n-1}) = |E|-1$.
  \end{enumerate}\label{th:dim}
\end{theorem}

\begin{proof}
  $P_{n,n-1}$   contains  exactly   $|E|$   integer  solutions   which
  corresponds   to  all   the   partitions  with   exactly  one   edge
  activated.  These  solutions  are  independent and  thus  $dim(P_{n,
    n-1})= |E|-1$.

  We now consider  $K< n-1$. Assume that $P_{n,K}$ is
  included in $H=\{ x\in\mathbb  R^{|E|+|V|-3} |$ $\al^T x = \al_0\}$.
  We prove that  $H$ is unique if $K=2$ and  that all its coefficients
  are  equal to  0  if  $K\in\{3, 4,  \hdots,  n-2\}$.  Unless  stated
  otherwise, the $K-$partitions considered throughout the remainder of
  this proof only require that two of their clusters contain more than
  one  element.  They are  thus feasible  since $K$  is assumed  to be
  lower than $n-1$. \bigskip

    We first  apply Lemma~\ref{lemma1} with $r_1=1$, $r_2=2$
    and  $c_1, c_2\in\{3, 4,  \hdots, n\}$  to obtain  that $\al_{c_1,
      c_2}=0$.   Similarly,  if  $r_2=3$  and $c_1=2$,  we  get  that
    $\al_{2,c_2}=0$ for all $c_2\geq 4$.  \bigskip

    Furthermore,  if  $\{1,3\}\subset   C_1$,  $\{2\}\subset  C_2$  and
    $c_2\in\{4,  5,   \hdots,  n\}$,  Lemma~\ref{lemma2}   states:
    $\al_{1,c_2}=0$.  Up to this level, we know that $\al_{i,j}=0$ for
    all $ij\in E\backslash \{12, 13, 23\}$.  \bigskip

    Lemma~\ref{lemma5} for $r_1=1$, $i=2$ and $r_2=3$ shows that
    $\al_{1,3}$  and $\al_{2,3}$  are equal  to a  value that  will be
    referred to as  $\beta$.  \bigskip

    Moreover,   if    $r_1=1$,   $r_2=2$,   $c_2=3$   and
    $c_1\in\{4,  5, \hdots,  n\}$, Lemma~\ref{lemma2}  can be  used to
    highlight    that    $2\beta+\al_{c_1}=\al_3$.    As   
    previously stated, the variable $x_3$ has been removed from the
    formulation and its corresponding  coefficient $\al_3$ is equal to
    $0$.   We  obtain  that  for  all  $s\in\{4,  5,
    \hdots, n\}$, $\al_{s}=-2\beta$.  \bigskip

    We now use Lemma~\ref{lemma4} with $\{1, 3\}\subset
    C_1$, $r_2=2$  and $r'_2\in\{4, 5,  \hdots, n\}$ to  get:   $2\al_{1,2}+\al_{r'_2}=0$.   We   then
    conclude that $\al_{1,2}=\beta$.  \bigskip

    At this point,  we know that $\al_{1,2}=\al_{1,3}=\al_{1,3}=\beta$,
    that for all $s\in\{4, 5, \hdots, n\}$ $\al_s=-2\beta$ and that all the
    other coefficients are equal to $0$. Thus, the only hyperplane
    which can contain $P_{n,K}$        is
    \begin{equation}
    \beta(x_{1,2}+x_{1,3}+x_{2,3}-2\zsum_{s=4}^nx_s)=\al_0.\label{eq:dim22}
  \end{equation}


  Each cluster  which does not contain  the vertices $1$, $2$  or $3$ has
  exactly one vertex whose representative variable has value  one.
  Moreover, if  $K=2$, the three first  vertices can either  be linked or
  scattered in the two clusters. In both cases, it can be checked that
  equation~(\ref{eq:dim22})       is      always       satisfied      if
  $\al_0=\beta$.  As a  consequence, exactly  one hyperplane
  contains $P_{n,2}$. Its dimension is thus $|E|+n-4$.\bigskip

    If $K\in\{3, 4, \hdots, n-2\}$ Lemma~\ref{lemma4} is used with
    $r_1=1$, $r_2=2$ and $3\in C_3$ to show
    that $2\beta - 4\beta=0$. $\beta$ is, therefore, equal to $0$.  In
    the general case, we conclude that there  is no  hyperplane which
    contains $P_{n, K}$. Therefore, its dimension is maximal.



\end{proof}

Thereafter, we study facet-defining  inequalities for $P_{n,K}$
when it is full-dimensional (\textit{i.e.:} $K\in\s{3,\hdots, n-2}$).
For  each studied  face $F=\s{x\in  P_{n,K}|\omega^T  x=\omega_0}$, we
consider  a facet-defining  inequality $\al^T  x\leq\al_0$ such  that
$F\subseteq \s{x\in P_{n,K}|\al^Tx = \al_0}$.  We then prove that $F$ is
facet-defining by highlighting, with reference to Theorem 3.6 in Section
I.4.3 of \cite{nemhauser1988integer},  that  $(\al,\al_0)$ is proportional
to $(\omega,\omega_0)$.

\section{Trivial facets}
\label{sec:trivial}

In this section, we show which of the inequalities from the integer formulation
are facet-defining. We  restrict our study to the  general cases where
$P_{n,K}$  is  full-dimensional  (\textit{i.e.}: $K\in\{3,  4,  \hdots,
n-2\}$).

\subsection{Edge bound inequalities}

\begin{remark} The inequalities $x_{u,v}\leq  1$ for all $uv\in E$ are
not  facet-defining  since  they  are  induced by  the  two  following
inequalities:  $x_{u,v}+x_{u,i}- x_{v,i}\leq 1$  and $x_{u,v}+x_{v,i}-
x_{u,i}\leq 1$ for all $i\in V\backslash\{u,v\}$.
\end{remark}

\begin{theorem}  If  $P_{n,K}$  is full-dimensional,  the  inequalities
$x_{u,v}\geq 0$ are facet-defining  if and only if $uv\not\in\{12, 13,
23\}$.
  \label{th:uv}
\end{theorem}

\begin{proof}  We  first  show  that  if $uv\in\{12,  13,  23\}$  then
$x_{u,v}\geq  0$ is not  facet-defining. If  $x_{1,2}=0$, vertex  $3$ is
either linked to $1$, $2$ or none of them.  The sum of representatives
$\sum_{i=4}^nx_i$ is  equal to  $K-2$ for the  two first cases  and to
$K-3$ for  the last  one.  We  then deduce that  the face  of $P_{n,K}$
defined by $x_{1,2}\geq 0$ is not a facet since it is also included in
the  hyperplane defined  by  $\sum_{i=4}^nx_i =  x_{1,3}+x_{2,3}+K-3$.
Symmetric   reasonings  yield  similar   results  for   $x_{1,3}$  and
$x_{2,3}$.

  We now consider  an edge $uv\in E$ such that $v\geq  4$ and we show:
$F_{u,v}=\{x^\pi\in  P_{n,K}|x^\pi_{u,v}=0\}$ is  a facet  of $P_{n,K}$.
To that end, we  consider a hyperplane $H=\{ x\in\mathbb R^{|E|+|V|-3}
| \;\al^T x = \al_0\}$ which  includes $F_{u,v}$ and we prove that all
its coefficients with the exception  of $\al_{u,v}$ are equal to zero.
$H=\{ x\in\mathbb R^{|E|+|V|-3}  | \;\al^T x = \al_0\}$  and 
that - except $\al_{u,v}$ - all the coefficients of $\al$ are equal to
zero.

The transformations used in the first part of Theorem~\ref{th:dim} can
be  used  again  here  by  setting $C_3=\{v\}$  in  order  to  ensure:
$x_{u,v}=0$. We thus obtain:
  \begin{itemize}
  \item $\al_{i,j}=0\,$ $ \forall ij\in E\backslash\{12,13,23\}, i\neq
v, j\neq v$;
  \item $\al_{1,2}=\al_{1,3}=\al_{2,3}\zdef\beta$;
  \item $\al_i=-2\beta\,$ $\forall i\neq v$.
  \end{itemize} Then we apply  Lemma~\ref{lemma1} with $r_1$ and $r_2$
in  $\s{1,2,3}\backslash\s{u}$, $v\in C_1,  C_3=\{u\}$ and  $i\in C_2$
with $i\in\{4,5,\hdots,  n\}\backslash\s{u,v}$ to deduce: $\al_{i,v}=0$.
It   remains   to  prove   that   $\al_{a,v}=\al_v=\beta=0$  for   all
$a\in\{1,2,3\}\backslash \s{u}$.

Let      $b\in\{1,2,3\}\backslash\{a,u\}$.     The     transformations
$\tr{\{a,b,u\}}{v}{a}$  and $\tr{\{a,v\}}{u}{a}$ (Figure~\ref{fig:uv1}
and~\ref{fig:uv2})   lead    to   $\beta=0$   and   $\al_{a,v}=\al_v$.
Eventually,     the     transformation    $\tr{\{a,b,v\}}{u}{\{a,b\}}$
(Figure~\ref{fig:uv3})   gives  $2\al_{a,v}=\al_v$   which   leads  to
$\al_{a,v}=\al_v=0$.

\end{proof}

\begin{minipage}{1.0\linewidth}
  \begin{center}
    \captionsetup{justification=centering, singlelinecheck=false }
    \begin{minipage}{0.3\linewidth}              \centering\trtikz{$b\;
u$}{$a$}{1}{}{$v$}{}{1.5}\\
    \end{minipage}
    \begin{minipage}{0.3\linewidth}
\centering\trtikz{$v$}{$a$}{1}{}{$u$}{}{1.5}\\
    \end{minipage}
    \begin{minipage}{0.3\linewidth}         \centering\trtikz{$v$}{$a\;
b$}{1}{}{$u$}{}{1.8}\\
    \end{minipage} \vspace{0.2cm}

    \begin{minipage}{0.3\linewidth}
      \captionof{figure}{$\tr{\{a,b,u\}}{v}{a}$}
      \label{fig:uv1}
    \end{minipage}
    \begin{minipage}{0.3\linewidth}
      \captionof{figure}{$\tr{\{a,v\}}{u}{a}$}
      \label{fig:uv2}
    \end{minipage}
    \begin{minipage}{0.3\linewidth}
      \captionof{figure}{$\tr{\{a,b,v\}}{u}{a}$}
      \label{fig:uv3}
    \end{minipage}
  \end{center}
\end{minipage}

\subsection{Representative bound inequalities}

\begin{remark}  The  inequalities $x_v\leq  1$  for  all $v\in\{4,  5,
\hdots, n\}$ are not facet-defining  since the face induced by $x_v=1$
is  contained in  the  hyperplanes $\{x\in\mathbb  R^{|E|+|V|-3}|x_{u,
v}=0\}$ for  all $u\in\{1,  2, \hdots, v-1\}$.  Indeed, if $v$  is the
representative of  a cluster $C$, it  must be the lowest  vertex of $C$.
The dimension of the face induced by $x_v\leq 1$ is thus lower than or
equal to $dim(P_{n, K})-v+1$.
\end{remark}

\begin{remark} For the same reason, the inequality $\sum_{i=4}^n x_i -
x_{1,2}\geq   K-3$   which  corresponds   to   $x_3\leq   1$  is   not
facet-defining.   Neither  is   $x_3\geq   0$  since   in  that   case
$x_{1,3}+x_{2,3}-x_{1,2}  =  1$  (\textit{i.e.}:  vertex $3$  is  not  a
representative so it is either with $1$, $2$ or both).
\end{remark}

\begin{theorem}  If  $P_{n,K}$  is full-dimensional,  the  inequalities
$x_v\geq 0$ for  all $v\in\{4, 5, \hdots, n\}$,  are facet-defining if
and only if $K\neq n-2$.
  \label{th:u}
\end{theorem}

\begin{proof}  We  first  prove  that, if  $K=n-2$,  the  inequalities
$x_v\geq 0$ are not facet-defining.   In that case, only two vertices are
not  representative and  the  face induced  by  $x_v \geq  0$ is  thus
included   in  the   hyperplane   defined  by   $\sum_{i=4}^n  x_i   =
x_{1,2}+x_{1,3}+x_{2,3}+K-3$.  Indeed,  the sum of  the representative
variables $\sum_{i=4}^n  x_v$ can, vary  from $K-3$ (if each  $1$, $2$
and $3$ is  in a cluster reduced  to one vertex) to $K-1$  (if the three
first vertices are in the same cluster).

Given a  valid partition $\pi=\{C_1, C_2, \hdots,C_K\}$  we know since
$K\leq n-3$ that at least three vertices of $V$ are not representative of
their  cluster.  As a  consequence, the  partitions considered  in the
first case of the proof  of Theorem~\ref{th:dim} are still valid if we
add vertex $v$ to $C_1$ or $C_2$.  Moreover, it is easy to check that we
can always add $v$ to $C_1$ or $C_2$ in such a way that $v$ is never a
representative of its cluster (it is always in a cluster with at least
one vertex  among $1$, $2$ and  $3$).  Therefore, by  following the same
reasoning we obtain  that if the face of  $P_{n,K}$ defined by $x_v\geq
0$  is included  in  a hyperplane  $H=\{  x\in\mathbb R^{|E|+|V|-3}  |
\;\al^T x = \al_0\}$ its  only non-zero coefficient is $\al_v$.  Thus,
the only hyperplane which contains the face is $x_v=0$.

\end{proof}

\subsection{Upper representative inequalities}

\begin{theorem}  If  $P_{n,K}$  is  full-dimensional,  the  inequalities
$x_{u,v}  +  x_{v}\leq  1$  for  all  $v\geq  4$  and  all  $u<v$  are
facet-defining if and only if $n\geq 6$ or $\{u,v\}\neq \{4,5\}$.
\end{theorem}

\begin{proof}$P_{n,K}$ is full-dimensional if and only if $K\in\s{3, 4,
\hdots, n-2}$.  Therefore, $n$ is greater than four and if it is equal
to five, $K$ must be equal to  three. As a result, if $n=5$, $u=4$ and
$v=5$,  every  $K-$partition  which  satisfies $x_{4,5}+x_5  =  1$  is
contained  in  the hyperplane  induced  by $2(x_4+x_5)+\sum_{i\leq  3}
x_{i,5} = x_{1,2}+x_{1,3}+x_{2,3}+1$.

In   other   cases,   we   assume  that   the   face   $F_{u,v}=\{x\in
P_{n,K}~|~x_{u,v}+x_v=1\}$  is  included in  a  hyperplane induced  by:
$\al^T x =  \al_0$.  To start with, we deduce,  similarly to the proof
of   Theorem~\ref{th:uv},    that   $\al_{i,j}=0\,$   $\forall   ij\in
E\backslash\{12,13,23\}$    with   $i\neq    v$    and   $j\neq    v$;
$\al_{1,2}=\al_{1,3}=\al_{2,3}\zdef\beta$;    $\al_i=-2\beta\,$   $\forall
i\neq   v$.   The   transformation   $\trd{u,v,i}{j}{i}$  with   $i\in
V\backslash\s{u,v}$   and   $j\in  \s{1,2,3}\backslash\s{i,u}$   gives
$\al_{i,v}=0$.

      We now consider a partition $\pi=\{C_1, C_2, \hdots, C_K\}$ such
that:
      \begin{itemize}
      \item $C_1=\{a,b\}$,  with $a$ and $b$ two  distinct vertices lower
than 4;
      \item $C_2=i$,  with $i\geq  4$ and different  from $u$  and $v$
(which is always possible whether $n\geq 6$ or $\{u,v\}\neq \{4,5\}$);
      \item $\s{u,v}\subset C_3$.
      \end{itemize}

The    transformation
$\tr{\{a,b\}}{i}{a}$  (Figure~\ref{fig:u+uv})  shows  that $\beta$  is
equal to zero. Eventually, the transformation $\tr{\{a,u\}} v u$ leads
to $\al_v=\al_{u,v}$ which concludes this proof.

\begin{figure}[h]
  \centering
  \trtikz{$b$}{$a$}{1}{}{$i$}{}{1.5}
  \caption{$\tr{\{a,b\}} i a$}
  \label{fig:u+uv}
\end{figure}

\end{proof}

\begin{theorem}  If  $P_{n,K}$  is  full-dimensional  the  inequalities
$x_{1,2}+x_{a,3}-\sum_{i=4}^n x_i \leq  3-K$ for $a\in\{1,2\}$ - which
correspond to $x_{a,3}+x_3\leq 1$ - are facet-defining.
  \label{th:3-1}
\end{theorem}

\begin{proof}    We    assume     that    the    face    $F_{u}=\{x\in
P_{n,K}|x_{1,2}-\sum_{i=4}^n   x_i  +x_{a,3}=3,K\}$   of   $P_{n,K}$  is
included in a hyperplane $H=\{ x\in\mathbb R^{|E|+|V|-3} | \;\al^T x =
\al_0\}$    and    we    show    that    $\al^Tx$    is    equal    to
$\al_{1,2}x_{1,2}+\al_{a,3}x_{a,3}-\sum_{i=4}^n  \al_ix_i $.   Let $b$
be  whichever  vertex,  $1$ or  $2$,  is  different  from $a$.   We  use
Lemma~\ref{lemma1} with  $r_1=a$, $r_2=3$  and $b\in C_3$  to show
$\al_{i,j}=0\,$ $ \forall i,j\geq  4$.  We then use Lemma~\ref{lemma4}
with    $r_1=a$,   $r_2=3$,    and   $b\in    C_3$    to   obtain:
$2\al_{a,3}+\al_i+\al_j=0\,$ $\forall i,j\geq  4$. We thus deduce that
the coefficients $\al_i$  are all equal to $-\al_{a,3}$.  We then show
$\al_{1,2}=\al_{a,3}$  by using  Lemma~\ref{lemma4}  with $1\in  C_1$,
$2\in C_2$ and $3\in C_3$.

  We    now   have   to    show   $\al_{c,i}=\al_{d,i}=0$    for   all
$c,d\in\s{1,2,3}$ and  $i\geq 4$. The transformation  $\tr{\{c, i\}} d
i$ shows $\al_{c,i}= \al_{d,i}$. Then the transformation $\tr{\{a,3\}}
i a$ with $i\geq 4$ yields $\al_{a,i}=0$.

  Eventually,   we  prove  that   $\al_{b,3}=0$  by   considering  the
transformation $\tr{\{1,2,3\}} i 3$ with $i\geq 4$.
\end{proof}

\subsection{Lower representative inequalities}

\begin{theorem}  If  $P_{n,K}$  is  full-dimensional  the  inequalities
$x_{u}  +  \sum_{i=1}^{u-1}x_{i,u}\geq1$  for   all  $u  \geq  4$  are
facet-defining.
\end{theorem}

\begin{proof}   We   assume   that   the   face   $F_{u}=\{x\in
P_{n,K}|x_{u}   +  \sum_{i=1}^{u-1}x_{i,u}=1\}$   is   included  in   a
hyperplane $H=\{ x\in\mathbb R^{|E|+|V|-3}  | \;\al^T x = \al_0\}$ and
we show $\al^Tx=\al_ux_{u} + \sum_{i=1}^{u-1}\al_{i,u}x_{i,u}$.

  We again  deduce similarly to the proof  of Theorem~\ref{th:uv} that
$\al_{i,j}=0\,$  $\forall  i,j\in\{ij\in E\backslash\{12,13,23\}|i\neq
u,       j\neq       u\}$,      $\al_{1,2}=\al_{1,3}=\al_{2,3}\zdef\beta$,
$\al_i=-2\beta\,$ $\forall i\neq u$.

  Let $c$  and $d$ be  two vertices lower  than $u$. It is  then possible
through the transformation $\tr{\{c, u\}} d u$ (Figure~\ref{fig:u+s1})
to show that $\al_{c,u}$ and $\al_{d,u}$  are equal to a value that we
denote  as  $\gamma$.   The  two   transformations  $\tr{\{1,2\}}   u   2$  and
$\tr{\{1,2,3\}} u 2$  (Figures~\ref{fig:u+s2} and \ref{fig:u+s3}) lead
respectively  to  $\beta  +  \al_u  =\gamma$ and  $2\beta  +  \al_u  =
\gamma$. Therefore, $\beta=0$ and $\al_u=\gamma$.

  \begin{center}

\captionsetup{justification=centering, singlelinecheck=false }

\begin{minipage}{1.0\linewidth}
  \begin{minipage}{0.3\linewidth}
\centering\trtikz{$c$}{$u$}{1}{}{$d$}{}{1.5}\\
  \end{minipage}
  \begin{minipage}{0.3\linewidth}
\centering\trtikz{$1$}{$2$}{1}{}{$u$}{}{1.5}\\
  \end{minipage}
  \begin{minipage}{0.3\linewidth}                \centering\trtikz{$1\;
3$}{$2$}{1}{}{$u$}{}{1.5}\\
  \end{minipage} \vspace{0.2cm}

  \begin{minipage}{0.3\linewidth}
    \captionof{figure}{ $\tr{\{c, u\}} d u$}
    \label{fig:u+s1}
  \end{minipage}
  \begin{minipage}{0.3\linewidth}
    \captionof{figure}{$\tr{\{1,2\}} u 2$}
    \label{fig:u+s2}
  \end{minipage}
  \begin{minipage}{0.3\linewidth}
    \captionof{figure}{ $\tr{\{1,2,3\}} u 2$}
    \label{fig:u+s3}
  \end{minipage}
\end{minipage}
\end{center}

  We eventually have  to prove that $\al_{i,u}=0$ for  all $i$ greater
than $u$ thanks to the transformation $\tr{\{c,u\}}{\{d,i\}} u$.
\end{proof}

\begin{theorem}  If   $P_{n,K}$  is  full-dimensional   the  inequality
$x_{1,2}+x_{1,3}+x_{2,3}-\sum_{i=4}^n x_i  \geq 3-K$ for $a\in\{1,2\}$
-   which   corresponds    to   $x_3+x_{1,3}+x_{2,3}\leq   1$   -   is
facet-defining.
  \label{th:3-2}
\end{theorem}

\begin{proof} The proof is  similar to the one of Theorem~\ref{th:3-1}
by considering that $a$ is either  1 or 2. The only difference is that
the last  transformation ($\mathcal T(\s{1,2,3}$  $\s{i} \s{3})$) does
not have to be considered.
\end{proof}

\subsection{Triangle inequalities}

\begin{theorem}  If   $P_{n,K}$  is  full-dimensional   the  inequality
$x_{s,t_1}+x_{s,t_2}-x_{t_1,t_2}\leq  1$ for  $s,t_1,t_2$  distinct in
$V$  is facet-defining  if and  only if  the following  conditions are
satisfied
  \begin{enumerate}[(i)]
  \item $s<t_1$ or $s<t_2$;
  \item $\{s,t_1,t_2\}\neq\{1,2,3\}$.
  \end{enumerate}
  \label{th:triangle}
\end{theorem}

Since the triangle inequalities are  a special case of the 2-partition
inequalities  the  reader can  refer  to  Theorem~\ref{th:st} for  the
proof.

\section{$2-$chorded cycle inequalities}
\label{sec:chorded}

In   this  section  we   address  the   $2-$chorded  cycle   class  of
inequalities,  first  introduced  in \cite{grotschel1990facets}.   Let
$C=\{e_1,   \hdots,  e_{|C|}\}$   be  a   cycle  in   $E$   such  that
$e_i=c_{i}c_{i+1}$  for  all $i$  in  $\{1,  2,  \hdots, |C|-1\}$  and
$e_{|C|}=c_1c_{|C|}$.   Let $V_C\zdef\{c_1,  c_2, \hdots,  c_{|C|}\}$ and
$U\zdef V\backslash V_C$.  Let $u_1,u_2,\hdots,u_{|U|}$ be the vertices of
$U$ ordered such that $u_1<u_2<\hdots<u_{|U|}$. From now on all the indices used on a vertex in
$V_C$ are given modulo $|C|$ (\textit{e.g.}  $c_{|C|+2}$ corresponds to
$c_2$).   The  set  of  $2-$chords  of  $C$  is  defined  as  $\overline
C=\{c_ic_{i+2}\in   E|i=1,  \hdots,   |C|\}$.  The   $2-$chorded  cycle
inequality induced by  a given cycle $C$ of length at  least 5 and its
corresponding $\overline C$ is defined as
\begin{equation}
  \label{eq:tc}
  x(E(C))-x(E(\overline C))\leq \lfloor \frac{1}{2}|C|\rfloor.
\end{equation}

We skip the proof of the following lemma. The reader can refer
to~\cite{grotschel1990facets} for further details.

\begin{lemma}
  The $2-$chorded cycle inequality  (\ref{eq:tc}) induced by a cycle $C$
  of length at least $5$ is valid for $P_{n,K}$. The corresponding face
  $F_C$ is not facet-defining if $|C|$ is even.
\end{lemma}

\begin{theorem}
The face $F_C$ induced by an odd cycle $C$ of size
 $2p+1$  is facet-defining  if the  following conditions
 are satisfied:
  \begin{enumerate}[(i)]
  \item $P_{n,K}$ is full-dimensional;
  \item $|U\cap\{1,2,3\}|\geq 2$;
  \item $2\leq p\leq n-K-|U\cap\{1,2,3\}|$;
  \item $K \geq 4$.
  \end{enumerate}

\end{theorem}

\begin{proof}
  We first highlight  $K-$partitions $\pi_i=\s{C_1,  C_2, \hdots,
    C_K}$ in $F_C$  whenever conditions (i) to (iv) are
  satisfied.

Let $c_i$ be a vertex in $V_C$. The $K-$partition $\pi_i$ is constructed as follows:
\begin{itemize}
\item The first cluster contains $c_i$, $u_1$ and $u_2$;
\item  The $2p$  remaining  vertices of  $V_C$  are scattered  in the  next
  clusters such that $p$ edges of $C$ are activated and none of $\overline
  C$. If cluster $K$ is reached the vertices are distributed in $C_{K-1}$
  and $C_K$ (see example Figure~\ref{fig:pitcc}): 
  \begin{itemize}
  \item $C_q=\s{ c_{i+2q-1},c_{i+2q}}\,$ $\forall q\in\s{2, \hdots,\min(p,K)}$;
  \item $C_{K-1}\supset \s{c_{i+K+2q-1},c_{i+K+2q}},$ for all
    $q$ odd natural number
    such that $K+2q\leq 2p;$
  \item $C_{K}\supset \s{c_{i+K+2q-1},  c_{i+K+2q}},$ for all $q$ even
    natural number
    such that $K+2q\leq 2p.$
  \end{itemize}

\item The remaining elements of $U$ are scattered in the next clusters
  or in $C_K$:
  \begin{itemize}
  \item $C_{p+q-1} = \s{u_{q}}\,$ $\forall q\in\s{3,K-p+1};$
  \item $C_K\supset\s{u_{q}}\,$ $\forall q\in\s{K-p+2, |U|}.$
  \end{itemize}
\end{itemize}

\begin{figure}[h]
  \centering
  \input{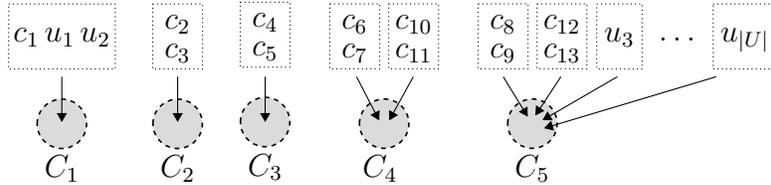}
  
  \caption{Construction of $\pi_1$ if $K=5$ and $|C|=13$.}
\label{fig:pitcc}
\end{figure}

From   this  construction  we   directly  deduce   that  $\pi_i$   is  a
$K-$partition. Since  no edge  of $\overline C$  is activated  and $p$
edges of
$C$ are, $\pi_i$ is also in $F_C$.


  All the $K-$ partitions considered throughout  this proof
  can  be obtained  from $\pi$  thanks  to  valid  transformations
  (\textit{i.e.} transformations  which lead  to a $K-$partition  which is
  still in $F_C$).

Assume that $F_C$ is included in  the hyperplane induced by $\al^T x =
\al_0$. We know from condition (ii) that $u_1$ and
$u_2$ are in $\s{1,2,3}$.

We  first show  that  for  all $i$  in  $\{1,  2, \hdots,  |C|\}$:
 $\al_{\ci, \ciu}=-\al_{\ci, \cid}\zdef\beta$.  To that end we consider the 
transformations represented Figures~\ref{fig:tc1}
and~\ref{fig:tc2}.  These transformations are valid since they do not
alter  the   number  of   clusters  and  the   sum  $x(E(C))-x(E(\overline
C))$.  Through  these  transformations  no  representative  variable  is
modified since $u_1$ and $u_2$ are lower than or equal to $3$.  We obtain
$\al_{\ud,\ci}=\al_{\ci,  \ciu}+\al_{\ci,  \cid}+\al_{\uu,  \ci}$  and
$\al_{\uu,\ci}=\al_{\ci,  \ciu}+\al_{\ci,  \cid}+\al_{\ud,  \ci}$.  We
deduce from  them:
\begin{equation}
\lambda_i\zdef\al_{\uu, \ci}=\al_{\ud,  \ci}\label{eq:tc1},
\end{equation}
and
\begin{equation}
\beta_i\zdef\al_{\ci,  \ciu}=-\al_{\ci, \cid}\label{eq:tc2}.
\end{equation}

If we substitute  $\ci$ and $\cid$ in the  transformations we get that
$\beta_{i+1}=-\al_{\ci,  \cid}=\beta_i$. By  applying  similarly these
two transformations on  every possible values of $i$  we conclude that
all the $\beta_i$ are equal to a value that we denote by $\beta$.

\begin{minipage}{1.0\linewidth}
  \begin{center}
    \captionsetup{justification=centering, singlelinecheck=false }

    \begin{minipage}{0.44\linewidth}
      \centering\trtikz{$\ud$}{$\ci$}{1}{$\ciu$}{$\uu\;\cid$}{}{1.8}\\
    \end{minipage}
    \begin{minipage}{0.44\linewidth}
      \centering\trtikz{$\uu$}{$\ci$}{1}{$\ciu$}{$\ud\;\cid$}{}{1.8}\\
    \end{minipage}
    \vspace{0.2cm}

    \begin{minipage}{0.44\linewidth}
      \captionof{figure}{$\tr{\{\ci,  \ud\}}{\{\ciu,   \cid,  \uu\}  }
        {\ci}$}
      \label{fig:tc1}
    \end{minipage}
    \begin{minipage}{0.44\linewidth}
      \captionof{figure}{$\tr{\{\ci,  \uu\}}{\{\ciu,   \cid,  \ud\}  }
        {\ci}$}
      \label{fig:tc2}
    \end{minipage}
  \end{center}
\end{minipage}
We now  show that all the other coefficients  of $H$ are equal
to zero. 

The   transformations  $\tr{\{\uu,  \ci\}}{\{\ciu,   \cid\}}{\ciu}$  and
$\mathcal T(\{\uu, \ud, \ci\},\qquad$ $\{\ciu, \cid\},\s{\ciu})$ respectively show 
$\al_{\ciu}\mathbb{1}(\ciu<\cid)=\lambda_{i+1}+\al_{\cid}\mathbb{1}(\ciu<\cid)$
and
$\al_{\ciu}\mathbb{1}(\ciu<\cid)=2\lambda_{i+1}+\al_{\cid}\mathbb{1}(\ciu<\cid)$. Accordingly,
we have
\begin{equation}
\lambda_i=0.\label{eq:tc4}
\end{equation}

Therefore, 
$\al_{\ciu}\mathbb{1}(\ciu<\cid)=\al_{\cid}\mathbb{1}(\ciu<\cid)$. This enables to deduce that all
the representative variables from $V_C$  are equal to a constant that we
 call $\gamma$. The transformation $\tr{\ci}{\{\uu, \ud, \ciu, \cid\}}{\uu}$
then give $\al_{\uu, \ud} + \gamma = 0$. 

\begin{itemize}
\item If $|C\cap\{1,2,3\}|=1$ we directly have $\gamma = 0$ and thus $\al_{\uu, \ud}=0$.
\item If  $|C\cap\{1,2,3\}|=0$ there  exists $u_3$ in  $\{1,2,3\}\cap(U\backslash\{\uu, \ud\})$
  and  the   transformation  $\tr{\{\uu,  \ud,   u_3,  \ci\}}{\{\ciu,
    \cid\}}{u_1}$ yields the same result.
\end{itemize}

Let $t$ be an index in $\{3, 4, \hdots, 2p-3\}$. We now  prove that
for all $i$ in $\{1, 2, \hdots, |C|\}$:
$\al_{\ci,   c_{i+t}}=\al_{\ci,  c_{i+t+1}}=0$.   The  transformations
represented in Figure~\ref{fig:tc3} and ~\ref{fig:tc4} give 
\begin{equation}
  \label{eq:tc5}
  \al_{\ci, c_{i+t}}+\al_{\ci, c_{i+t+1}}=0,
\end{equation}
and
\begin{equation}
  \label{eq:tc6}
  \al_{\ci, c_{i+t}}+\al_{c_{i-1}, c_{i+t}}=0.
\end{equation}

\begin{minipage}{1.0\linewidth}
  \begin{center}
    \captionsetup{justification=centering, singlelinecheck=false }
    \begin{minipage}{0.46\linewidth}
      \centering\trtikz{$\uu$}{$\ci$}{1}{$c_{i+t}$}{$c_{i+t+1}$}{}{1.8}\\
    \end{minipage}
    \begin{minipage}{0.46\linewidth}
      \centering\trtikz{$\uu$}{$c_{i+t}$}{1}{$c_{i-1}$}{$\ci$}{}{1.8}\\
    \end{minipage}
    \vspace{0.2cm}

    \begin{minipage}{0.46\linewidth}
      \captionof{figure}{$\tr{\{\ci,   \uu\}}{\{c_{i+t},  c_{i+t+1}\}}
        {\ci}$}
      \label{fig:tc3}
    \end{minipage}
    \begin{minipage}{0.46\linewidth}
      \captionof{figure}{$\tr{\{c_{i+t},   \uu\}}{\{c_{i-1},  \ci\}  }
        {c_{i+t}}$}
      \label{fig:tc4}
    \end{minipage}
  \end{center}
\end{minipage}

From these two equations we conclude that there exists a variable
$\al_{odd}$ such that for each $i\in\s{1,2,\hdots, |C|}$ and
for all $t\in\{3, 4, \hdots, 2p-2\}$ 
\begin{equation}
  \label{eq:tc666}
  \al_{c_i, c_{i+t}}=\left \{
\begin{array}{rcl}
\al_{odd} & \mbox{if } t+i\mbox{ is odd}\\
-\al_{odd} & \mbox{otherwise}
\end{array}.
\right.
\end{equation}

Accordingly, we have $\al_{c_{2p-2}, c_{2p+1}}=\al_{odd}$.  However $c_{2p+1}=c_1$
since  $|C|=2p+1$.  As  a  result $\al_{c_{2p-2},  c_{2p+1}}=\al_{c_1,
  c_{2p-2}}$ and  from the fact that  $2p-2$ is even we  also get that
$\al_{c_{2p-2}, c_{2p+1}}=-\al_{odd}$. The coefficient $\al_{odd}$ is then
equal to zero.

To  finish the  proof  we show  that  if there  exists  $u$ in  $U\cap
(V\backslash\{1,2,3\})$ then  all the coefficients related  to $u$ are
equal     to    zero.      The     transformations    presented     in
Figures~\ref{fig:tc5},~\ref{fig:tc6}    and    \ref{fig:tc7}   yields,
respectively,    $\al_{\ci,    u}=0$,    $\al_{u_1,   u}=\al_u$    and
$\al_u\mathbb{1}(\ci<u)=0$.   If there exists  $\ci<u$ we  then obtain
that  $\al_u$  and $\al_{\uu,  u}$  are equal  to  zero.   If not  the
transformation $\tr{\{\ci, \uu,  \ud, u\}}{\{\ciu, \cid\}}{\uu}$ gives
the expected result.

\begin{minipage}{1.0\linewidth}
  \begin{center}
    \captionsetup{justification=centering, singlelinecheck=false }
    \begin{minipage}{0.43\linewidth}
      \centering\trtikz{$\uu\; u$}{$\ci$}{1}{$\ciu$}{$\cid$}{}{1.5}\\
    \end{minipage}
    \begin{minipage}{0.43\linewidth}
      \centering\trtikz{$u$}{$\uu, \ci$}{1}{$\ciu$}{$\cid$}{}{1.8}\\
    \end{minipage}
    \vspace{0.2cm}

    \begin{minipage}{0.43\linewidth}
      \captionof{figure}{  $\tr{\{\uu,   u,  \ci\}}  {\{\ciu,  \cid\}}
        {\ci}$}
      \label{fig:tc5}
    \end{minipage}
    \begin{minipage}{0.43\linewidth}
      \captionof{figure}{$\tr{\{\uu,  u, \ci\}}{\{\ciu, \cid\}}{\{\uu,
          \ci\}} $}
      \label{fig:tc6}
    \end{minipage}
    \begin{minipage}{0.43\linewidth}
      \centering\trtikz{$u$}{$\ci$}{1}{$\ciu$}{$\cid$}{}{1.8}\\
    \end{minipage}
    \vspace{0.2cm}

    \begin{minipage}{0.43\linewidth}
      \captionof{figure}{ $\tr{\{\ci, u\}}{\{\ciu, \cid\}}{\ci}$}
      \label{fig:tc7}
    \end{minipage}
  \end{center}
\end{minipage}

\end{proof}

\section{2-Partition inequalities}
\label{sec:2part}

This   section  is  dedicated   to  the   study  of   the  2-partition
inequalities, first  introduced in \cite{grotschel1990facets}  for the
general clique partitioning problem. For two disjoint nonempty subsets $S$ and $T$ of $V$ are defined as
\begin{equation}
  \label{eq:st}
  x(E(S),E(T))-x(E(S))-x(E(T))\leq \min(|S|,|T|).
\end{equation}

Let $F_{S,T}$  be the  face of
$P_{n,K}$ defined
by equation~\ref{eq:st}. 

The proof of the three following lemmas is skipped. For further details the reader may refer
to~\cite{grotschel1990facets} for Lemmas~\ref{le:st0} and~\ref{le:st1}
and  to
~\cite{deza1992clique} for Lemma~\ref{le:st2}.

\begin{lemma}
  Inequality~\eqref{eq:st} is valid for $P_{n,K}$.
\label{le:st0}
\end{lemma}

\begin{lemma}
  If $|S|=|T|$ $F_{S,T}$, is not a facet of $P_{n,K}$.
\label{le:st1}
\end{lemma}

\begin{lemma}
  Given two disjoint subsets $S$ and $T$ of $V$ such that $|S|<|T|$.  A
  $K-$partition  $\pi=\{C_1,  C_2,   \hdots,  C_K\}$  is  included  in
  $F_{S,T}$ if and only if for  all $i\in\{1, 2, \hdots, K\}$\; $|T\cap C_i|
    - |S\cap C_i| \in\{0,1\}$
\label{le:st2}
\end{lemma}

Let $U=\{u_1,  u_2, \hdots, u_{|U|}\}$ be  the set of  vertices defined by
$V\backslash (S\cup T)$ such that  $u_1<u_2<\hdots
<u_{|U|}$.  The elements of $S$ and $T$ -  $\{s_1, s_2, \hdots,
s_{|S|}\}$ and $\{t_1, t_2, \hdots, t_{|T|}\}$ - are similarly sorted.

\begin{theorem}
  If     $P_{n,K}$      is     full-dimensional,     the     2-partition
  inequality~(\ref{eq:st}) is facet-defining  for two non empty disjoint
  subsets $S$ and  $T$ of $V$ if and only  if the following conditions
  are satisfied:
  \begin{enumerate}[(i)]
  \item $|T|-|S|\in\{1, 2, \hdots, K-1\}$;
  \item $|S|\leq n-(K+2)$; 
  \item $\forall s\in S\,$ $\exists t\in T,\; t>s$;
  \item if $|S|=1\,$ $\exists u\in U \cap\{1,2,3\}$.
  \end{enumerate}
\label{th:st}  
\end{theorem}

 To ensure that the each transformation used throughout the proof of this
 theorem leads from one
$K-$partition   in  $F_{S,T}$  to   another,  we   present  sufficient
conditions on the two clusters involved.




\begin{lemma}
Let $C_1$  and $C_2$ be  two disjoint subsets  of $V$. There  exists a
$K-$partition in $F_{S,T}$ which contains $C_1$ and $C_2$ if
  \begin{enumerate}[(i)]
  \item $|C_1\cap T| - | C_1\cap S| = 0$;
  \item $|C_2\cap T| - | C_2\cap S| = 1$;
  \item $|(C_1\cup C_2)\cap(T\cup U)|\leq 4$
  \item $S$ and $T$ satisfy the conditions of Theorem~\ref{th:st}.
  \end{enumerate}
\end{lemma}

\begin{proof}

Let  $S_{1/2}$,  $T_{1/2}$  and  $U_{1/2}$  refer  to  the  vertices
included in $C_1$
and $C_2$ which  are respectively in $S$, $T$  and $U$ (\textit{i.e.:}
$S_{1/2}=S\cap(C_1\cup C_2)$).
Let $S'$ (resp. $T'$ and $U'$) correspond to the elements of
$S$ (resp. $T$ and $U$) which are not in $C_1$ or $C_2$ (\textit{i.e.:}
$S'=  S\backslash S_{1/2}$). We define  the elements  of $S'$,  $T'$ and  $U'$ as
follows:   $S'=\s{s'_1,   \hdots,
  s'_{|S'|}}$, $T'=\s{t'_1, \hdots, t'_{|T'|}}$ and $U'=\s{u'_1, \hdots, u'_{|U'|}}$.

We first define the partition $\pi$ 
as follows (see
example in Figure~\ref{fig:pib}):
\begin{itemize}
\item The first two clusters are $C_1$ and $C_2$;
\item The next $|T'|-|S'|$ clusters are reduced to one vertex of $T'$: $C_{i}=\{t'_{i-2}\}\,$ $\forall i\in\{3,4\hdots,
  |T'|-|S'|+2\}$;
\item The remaining vertices of $T'$ and $S'$ are scattered in the next
  clusters, or in $C_{K}$ if it is reached:
  \begin{itemize}
  \item    $C_{i}=\{s'_{i-|T'|+|S'|+2},    t'_{i-2}\}\,$ $\forall
  i\in\{|T'|-|S'|+3, \hdots, \min(|T'|+2, K-1)\}$;
\item $C_{K}\supset\{s'_{i-|T'|+|S'|+2}, t'_{i-2}\}\,$ $\forall
  i\in\{K, |T'|+2\}$;
  \end{itemize}
\item  The  elements  of  $U'$  are similarly  scattered  in  the  next
  clusters or in $C_{K}$:
  \begin{itemize}
  \item    $C_{i}=\{u'_{i-|T'|-2}\}\,$ $\forall
  i\in\{|T'|+3, \hdots, \min(|T'|+|U'|+2, K-1)\}$;
\item $C_{K}\supset\{u'_{i-|T'|-2}\}\,$ $\forall
  i\in\{K, |T'|+|U'|+2\}$;
  \end{itemize}

\end{itemize}

\begin{figure}[h!]
  \centering\scalebox{0.85}{\input{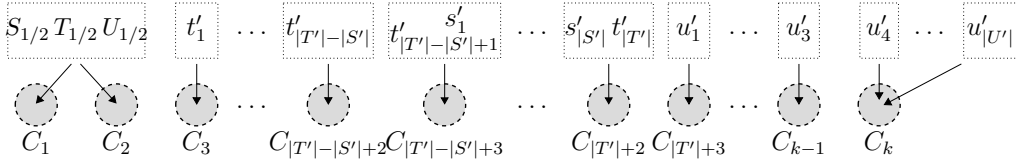}}

  \caption{Construction of $\pi$ if $K=|T'|+6$ and $|U'|\geq 4$.}
\label{fig:pib}
\end{figure}

If the obtained partition $\pi$ is a
  $K-$partition,   Lemma~\ref{le:st2}   ensures    that   it   is   in
    $F_{S,T}$. We must therefore prove that $|T'|+|U'|+2$ is greater than or equal
to $K$. 

According to the second condition of Theorem~\ref{th:st}, $K$ is lower
than or equal  to $n-(|S|+2)$.  Due to the fact  that $n-|S|$ is equal
to  $|T|+|U|$, we  obtain that  $K$  must be  lower than  or equal  to
$|T_{1/2}|+|T'|+|U_{1/2}|+|U'|-2$ .   From the third  condition of the
current lemma, we  deduce that $|T'|+|U'|+2$ is greater  than or equal
to $K$.

We then conclude from Lemma~\ref{le:st2} that $\pi$ is in $F_{S,T}$. 
\end{proof}

Each transformation considered in the  proof of the theorem involves a
couple  of clusters  which satisfies  the conditions  of  the previous
lemma  before  and after  the  transformation.   Hence, the  relations
highlighted by  the transformations are satisfied  by the coefficients
of any hyperplane which contains $F_{S,T}$.

\begin{lemma}
 If $S$  and $T$ fulfill  the conditions of  Theorem~\ref{th:st}, each
 hyperplan $H=\{ x\in\mathbb
  R^{|E|+|V|-3} | \;\al^T x = \al_0\}$ which indludes $F_{S,T}$ satisfies:  $\forall u\in  U$ 
  $\forall v\in V\backslash\s{u}$  $\al_{u,v}=0$.
 \label{le:uv}
\end{lemma}

\begin{proof}
  To prove this lemma, we consider the  following cases:
  \begin{itemize}
  \item case 1: $\s{t_1,t_2}\subset\s{1,2,3}$;
  \item case 2: $\s{s_1, s_2}\subset\s{1,2,3}$;
  \item case 3: $\s{u_1, u_2}\subset\s{1,2,3}$;
  \item case 4: $\{s_1, t_1, u_1\}=\{1,2,3\}$.
  \end{itemize}

  In each of these cases, we prove that for all $s\in S,\; t\in T$ and
  $u'\in U$ the coefficients $\al_{s,u},\; \al_{t,u}$ and $\al_{u,u'}$
  are equal to zero.
\bigskip

\qquad\textit{\underline{Case 1}}: $\s{t_1,t_2}\subset\s{1,2,3}$

If the  two first vertices  of $T$ are  lower than $4$ the  results shown
table~\ref{tab:case1} yield  the major part of the  lemma. It remains
to prove that $\al_{t,u}= 0$ for all $t\in T$.

\begin{table}
  \centering

      \begin{tabular}{ccclp{3.4cm}@{}l}
        \toprule
        \textbf{Lemma} & $C_1$ & $C_2$ & \textbf{Result} & \\\midrule
        \ref{lemma1} & $\s{s, t_1}$ & $\s{t_2, u}$ & 
        $\al_{s,u}=0$ & $\forall s\in S$ & \tagarray\label{eq:lst1}\\
        \ref{lemma1} & $\s{s,t_1,u}$ & $\s{t_2,u'}$ & 
        $\al_{u,u'}=0$ & $\forall u'\in U$ & \tagarray\label{eq:lst2}\\
        \bottomrule
      \end{tabular}  
  \caption{Results obtained in the case $\s{t_1,t_2}\subset\s{1,2,3}$.}
  \label{tab:case1}
\end{table}

From condition (i)  and (ii) we can deduce  that $|T|\leq n-3$.  Thus,
we either have $|S|\geq 2$ or  $|U|\geq 2$. Let $C$ be a cluster equal
to $C=\s{s_2,t'}$ with $t'\in T\backslash\s{t_1,t}$ if $|S|\geq 2$ and
$C=\s{u'}$  with   $u'\in  U\backslash\s{u}$  if   $|U|\geq  2$.  The
transformations $\tr{\s{t_2, u}}{C}{\s{u}}$ and $\tr{\s{s_1, t_1, t_2,
    u}}{C}{\s{u}}$  lead  to two  identical  equality  except for  the
coefficients  $\al_{t,u}$  and   $\al_{s_1,u}$  which  appear  in  the
second. Since $\al_{s_1,u}$ is equal to zero the same applies to $\al_{t,u}$

Eventually,    we    obtain    $\al_{t_1,u}=0$   via    transformation
$\trd{s,t_1}{t_2,u}{u}$.\bigskip

\qquad\textit{\underline{Case 2}}: $\s{s_1, s_2}\subset\s{1,2,3}$

According to (i), since $|S|\geq 2$, $|T|\geq 3$.  Then the first line of table~\ref{tab:case2} proves that
$\al_{t,u}$ is  always equal to zero for  all $t\in T$.  Let  $s$ be a
vertex in $S\backslash\s{s_1}$. The transformations $\trd{s_1, t, u}{t'}{u}$ and
$\mathcal T(\s{s_1,s,t,t'',u},$ $\s{t'},\s{u})$ show that $\al_{s, u}$ is equal
to zero.  A symmetrical reasoning by substituing respectively $s_1$ and $s$ by
$s_2$ and $s_1$ leads to $\al_{s_1,u} = 0$. Eventually, equations~(\ref{eq:lst23})
and~(\ref{eq:lst24}) show that all the
coefficients $\al_{u,u'}$ are also equal to zero. 




\begin{table}
  \centering

      \begin{tabular}{ccccll@{}l}
        \toprule
        \textbf{Lemma} & $C_1$ & $C_2$ & \textbf{Result} & \\\midrule
        \ref{lemma1}  &  $\s{s_1,  t,
            t'}$ & $\s{s_2, t'', u}$ &
        $\al_{t,u}=0$ & $\forall  t\in T,\,$
        & \tagarray\label{eq:lst21}\\
        & & & & $t',t''\in T\backslash \s{t} $\\

        \ref{lemma5}  & $\s{u,t}$  & $\s{u'}$  & 
        $\al_{u,u'}=\al_{t,u'}$                      &
        $\forall u'\in U,\, $ &
        \tagarray\label{eq:lst23}\\
        & & & & $t<u',\, u<u'$\\
        \ref{lemma5}    &   $\s{s_1,   t,    u}$   &    $\s{t',u'}$   &
        $\al_{s_1,u'}=\al_{u,u'}$ & $\forall u'\in U,\, $ &
        \tagarray\label{eq:lst24}\\
        & & & & $u<u'<t'$\\
        \bottomrule
      \end{tabular}  
  \caption{Results obtained in the case $\s{s_1,s_2}\subset\s{1,2,3}$.}
  \label{tab:case2}
\end{table}

\bigskip

\qquad\textit{\underline{Case 3}}: $\s{u_1, u_2}\subset\s{1,2,3}$

Let $u$ be a vertex  of $U$ and $a\in \s{u_1,u_2}\backslash\s{u}$. The
transformations $\trd{a}{t,u}{u}$ and  $\trd{a,s,t}{t,u}{u}$ lead to :
$\al_{s,u}+\al_{t,u}=0$.       Then,      $\trd{s,t}{a,t'}{a}$     and
$\trd{s,t}{a,t',u}{a,u}$ show that $\al_{t,u}$ is equal to zero. Thus,
the same applies to $\al_{s,u}$.

We  now have  to prove  that  $\al_{u,u'}$ is  equal to  zero for  all
$u,u'\in  U$.  We  assume  wlog  that  $u$ is  lower  than  $u'$.   If
$\min(s_1,t_1)\leq 3$, $\trd{u,u'}{s_1,t_1,t_2}{u'}$ gives the result.
Otherwise,  $U$ necessarily  contains  at least  three vertices.   Let
$a\in U\backslash\s{u,u'}$.  We then conclude  through
$\trd{a,u}{t}{a}$ and $\trd{a,u,u'}{t}{a,u'}$.  \bigskip

\qquad\textit{\underline{Case 4}}: $\{s_1, t_1, u_1\}=\{1,2,3\}$

Table~\ref{tab:case4}  enable  to conclude  for  all the  coefficients
except $\al_{s,u_1}$ for all $s\in S$ and
$\al_{t,u_1}$ for all $t\in T$.

$\mathcal T(C\cup\s{t_1},\s{t},\s{t_1,t})$ with $C$ successively equal
to $\s{s_1}$ and $\s{s_1,u_1}$ shows that 
$\al_{t_1,u_1}$ and $\al_{t,u_1}$ are equal.  Then, $\trd{t_1,
  u_1}{s_1, t}{u_1}$ leads to: $\al_{s_1,u_1}=0$.

As previoulsy stated, at least one of the set $S$ ou $U$ contains more
than one vertex. If $U$ contains at least two vertices, $\mathcal
T(P\cup\s{s_1,t},\s{u_2},\s{s_1,t})$ with $P$ successively equal to $\s{t_1}$ and 
$\s{t_1,u_1}$ gives:  $\al_{t,u_1}=0$.  If $S$
contains more than one vertex, we consider
$\mathcal T(P\cup\s{s,t},\s{s',t'},\s{s,t})$ with $\s{s,s'}\subset S$,
$\s{t,t'}\subset T$  and $P$  a set of  vertices. By first  taking $s$
equal to $s_1$ and considering $P$ equal to $\s{t_1}$ and
$\s{t_1,u_1}$, we show  that $\al_{t,u_1}$ is equal to  zero. The same
reasoning with $s\in S\backslash\s{s_1}$ gives: $\al_{s_1,u_1}=0$.



\begin{table}
  \centering

      \begin{tabular}{cccclll}
        \toprule
        \textbf{Lemma} & $C_1$ & $C_2$ & \textbf{Result} & \\\midrule
        \ref{lemma1}  &   $\s{s_1,  t,  u}$   &  $\s{t_1,  u'}$   &
        $\al_{u,u'}=0$      &      $\forall      u'\in     U$      &
        \tagarray\label{eq:lst41}\\
        \ref{lemma1} & $\s{t,u_1,u}$  & $\s{t_1,s}$ & $\al_{s,u}=0$
        & $u\neq u_1$ & \tagarray\label{eq:lst42}\\
        \ref{lemma1}  &   $\s{u_1,  u}$  &  $\s{s_1,  t,   t'}$  &  
        $\al_{t,u}=0$ & $\forall t,t'\in T$, $u\neq u_1$ &
        \tagarray\label{eq:lst44}\\

        \bottomrule
      \end{tabular}  
  \caption{Results obtained in the case $\{s_1, t_1, u_1\}=\{1,2,3\}$.}
  \label{tab:case4}
\end{table}
\end{proof}

We now present the proof of Theorem~\ref{th:st}.

\begin{proof}

We
start by  proving that if any  of the conditions is  not satisfied the
2-partition inequality is not facet-defining.

  \begin{enumerate}[(i)]
  \item   Lemmas~\ref{le:st1} and~\ref{le:st2}  respectively
    lead to  $|T|-|S| >0$  and $|T|-|S|\leq K$.   It remains  to prove
    that if $|T|-|S|= K$, equation (\ref{eq:st}) is not facet-defining.  In that
    case,  each cluster  $C$ verifies,  $|C\cap T|  = |C\cap  S|+1$ and
    $F_{S,T}$ is thus included in the $|T|$ hyperplanes defined by
    $\sum_{i\in T\backslash \s{t}}x_{i,t}=\sum_{i\in S}x_{i,t}$ for all $t\in T$.

  \item We   get via Lemma~\ref{le:st2} that each cluster
    $C$   contains  at least  as  many  vertices  from $T$  than  from
    $S$. Thus, at least $|S|$ vertices are not a representative of their
    cluster,  and then  $K\leq n-|S|$.   If  $K=n-|S|$,
    $x(E(S),E(T))$ is equal to $|S|$ since the only edges in the partition are necessarily the
    ones which ensure that each vertex $s\in S$ is linked to a vertex in $T$.
    Eventually,  if  $K=n-|S|-1$  we   deduce  by  enumerating  all
    possible configurations that $F_{S,T}$ is included in the hyperplane
    defined by $x(E(U))+x(E(U),E(T))+x(E(S),E(T))-x(E(S))=3$.

  \item  Lemma~\ref{le:st2} states that each cluster contains
    at least  as much  vertices from $T$  than from  $S$. As a  result if
    there was a vertex $s$ in $S$  greater than any vertex in $T$ we would
    have $x_s=0$ since $s$ would never be a representative.

  \item  If $|S|=1$  and  there is  no  element  of $U$  in
    $\{1,2,3\}$ there is then at least two elements of $T$, $t_1$ and
    $t_2$, whose indices are lower than or equal to $3$. The last element in
    $\{1,2,3\}$ is either in $T$ or $S$.
    \begin{itemize}
    \item   If  it   is  in   $T$   we  have   $\sum_{i=4}^n  x_i   =
      x_{1,2}+x_{1,3}+x_{2,3}+K-3$;
    \item It it is in $S$ we have $\sum_{i=4}^n x_i = x_{s,t_1}+x_{s,t_2}+K-3$.
    \end{itemize}

\end{enumerate}

To prove that $F_{S,T}$ is  facet-defining under the above mentioned
conditions,  we  assume that  $F_{S,T}$  is  included in  a  hyperplane
$H=\{ x\in\mathbb R^{|E|+|V|-3} | \;$ $\al^T x = \al_0\}$. We 
highlight relations  between the $\al$ coefficients  thanks to several
transformations from one $K-$partition to another in $F_{S,T}$. 

We consider three cases.  
\begin{itemize}
\item case 1: $|S|=1$;
\item case 2: $|S|\geq 2$ and $|U|=0$;
\item case 3: $|S|\geq 2$ and $|U|\geq 1$.
\end{itemize}
\bigskip

\qquad\textit{\underline{Case 1}}: $|S|=1$

In that case, according to (iv), $u_1$ is in $\{1,2,3\}$.  From
(i) and (ii) we deduce that $|T|$ is lower than $n-3$. As a result, since $|S|=1$,  $|U|$ is greater than
$2$.

Let $t$ be either $t_1$ or $t_2$.  For all $t'\in T\backslash\s{t}$, by noticing that $\min\s{s_1, t,
  u_2}\leq 3$, we
deduce, from $\trd{s_1, t, u_2}{t', u_1}{u_1}$ that: $\al_{t'}=0$.

Thanks  to  (iii)  we   know  that  $s_1$  is  lower  than
$\ot$.  Moreover,  since  only  $u_1$  and $s_1$  may  be  lower  than
$\min(t_1,  u_2)$,  we know  that  $\al_{\min(t_1,u_2)}$  is equal  to
zero. Thus, $\mathcal T(\s{t_1, u_2},$ $\s{s_1,\ot, u_1},\s{u_1})$ gives:
$\al_{s_1}=0$. For all $u\in U\backslash \s{u_1}$, the transformation $\trd{u_1, u}{\ot}{u_1}$ proves that
the  coefficients $\al_{u}$  are   null.

Eventually, for all $t,t'\in T$, we prove that the expressions
$\al_{s_1,t}$ and $-\al_{t, t'}$  are equal through $\trd{s_1, t}{t'}{s_1}$
and $\trd{s_1, t, t'}{u_1}{t}$.
\bigskip

\qquad\textit{\underline{Case 2}}: $|S|\geq 2$ and $|U|=0$

Conditions (i) and (ii) still lead to  $|T|\leq n-3$ which
gives $|S|\geq 3$, in the current case.  Since $P_{n,K}$ is only full-dimensional for
values of $K$ greater than two, condition (ii) implies
that $|T|$ is greater than four.

In this part of the proof, let the expressions $\Ot$ and $\Tt$
both correspond to $t_1$ or  $t_2$ with the restriction that $\Ot$ and
$\Tt$ are distinct. Similarly, $\Os$ and $\Ts$ correspond to $s_1$ or
$s_2$.  Since  $|U|=0$, $\min(\Ot,  \Os)$ is lower  than four  and its
corresponding  representative variable  $\alpha_{\min(\Ot,  \Os)}$ is,
therefore, equal to zero.

We first prove $\forall s\in S\backslash \{s_1, s_2\}\;\forall t,t'\in
T\backslash  \{t_1,   t_2\}$  that  $\beta\zdef\al_{s,t}=   -  \al_{t,t'}=
\al_{\Os,\ot}=-\al_{\Ot,\ot}$.

We obtain $\al_{s,t}=-\al_{t,t'}$  thanks to the
transformations           represented           Figures~\ref{fig:st2.1}
and~\ref{fig:st2.2}. The remaining part of the equation is highlighted
via      transformations   $\tr{\s{s,t,\ot}}{\s{s',t'}}{\ot}$   and
$\tr{\s{\overline s, \overline t,s,t,\ot}}{\s{s',t'}}{\ot}$.

\begin{minipage}{1.0\linewidth}
  \begin{center}
    \captionsetup{justification=centering, singlelinecheck=false }
    \begin{minipage}{0.44\linewidth}
      \centering\trtikz{$s_1\; t_1$}{$t$}{1}{$t_2$}{$s_2$}{}{1.8}\\
    \end{minipage}
    \begin{minipage}{0.44\linewidth}
      \centering\trtikz{$s_1\; t_1\; s\;t'$}{$t$}{1}{$t_2$}{$s_2$}{}{1.8}\\
    \end{minipage}
    \vspace{0.2cm}

    \begin{minipage}{0.44\linewidth}
      \captionof{figure}{$\tr{\s{s_1, t_1, t}}{\s{s_2, t_2}} {t}$}
      \label{fig:st2.1}
    \end{minipage}
    \begin{minipage}{0.44\linewidth}
      \captionof{figure}{$\tr{\s{s_1,  t_1, s,  t,  t'}}{\s{s_2, t_2}}
        {t}$}
      \label{fig:st2.2}
    \end{minipage}
  \end{center}
\end{minipage}
For any couple of vertex $(s,t)\in S\times T$, we now consider the
transformations $\trd{\Os, s, \Ot,  t}{\ot}{s}$ and $\trd{\Os, s, \Ot,
  \Tt}{t}{\Ot, t}$ which respectively lead to 
\begin{equation}
  \label{eq:B}
  \al_{s,\Os}+\al_{s,\Ot}+\al_{s,t}+\al_\ot = \al_s+\beta
\end{equation}
and
\begin{equation}
  \label{eq:A}
  \al_{\Ot,                                              \Tt}+\al_{\Ts,
    \Ot}+\al_{s,\Ot}+\al_t=\al_{\Ot}+\al_{s,t}+\al_{\Ts,t}+\al_{\Tt, t}.
\end{equation}

Equation~(\ref{eq:B}) shows  that for any  given $s\in S$, the  value of
$\al_{s,t}$ is the same for all $t\in T\backslash \s{\ot}$. Since
$\al_{s,t_3}$    is    equal    to    $\beta$,   we    deduce:
\begin{equation}
\al_{s,\Ot}=\beta\,\mbox{ }\forall s\in S.\label{eq:C}
\end{equation}

After replacing  $\al_{s,\Ot}$ by $\beta$  in equation~(\ref{eq:A}), a
similar reasoning can be applied to prove:
$\al_{\Os, t}=\beta\,$ $\forall t\in T\backslash \s{t_1, t_2}$.  

  The
transformation   $\trd{\Os,   \Ts,   s,\Ot,   \Tt,   t}{\ot}{s}$   and
equation~(\ref{eq:B})   give:    $\al_{s,   \Os}=-\beta\,$ $\forall   s\in
S\backslash \s{s_1, s_2}$.

If  $t_3\leq 3$  we prove  by symmetry  that $\al_{\Ot,t}=\al_{t_3,t}\,$
$\forall      t\in     T\backslash\s{t_1,     t_2,      t_3}$.     Thus,
$\al_{\Ot,t}=\beta$.  Otherwise $s_1\leq  3$  and the  same result  is
obtained      via     equation~\eqref{eq:B}      and     transformations
$\trd{t,s_1}{\ot}{s_1}$ and $\trd{s,t,\Ot}{s,\ot}{t}$.


 Equation~(\ref{eq:A}) now shows that the value of the representative
coefficients $\al_t$ for any $t\in T\backslash \s{t_1,
  t_2}$ are the same.   Equation~(\ref{eq:B})
applied on any
$s\in S\backslash\s{s_1,  s_2}$ proves: $\al_s =  \al_t$. Let $\gamma$
be this value.

If $t_3$ or $s_3$ is lower than four, $\gamma$ is equal to zero (since
$\al_{\min(s_3,t_3)}$ is  null).  Otherwise, $s_1$ and $t_1$ must
be lower than four. Since $\mathcal T (\s{\Os, \Ot},\s{\ot},$ $\s{\Os})$ gives
$\gamma = \al_{\max\s{\Os, \Ot}}$ we deduce, in that case also, that $\al_t$
is equal to zero.  Since $\max\s{\Os, \Ot}$ and
$\min\s{\Os, \Ot}$  are both equal  to zero, we conclude  that $\al_{\Os}$
and $\al_{\Ot}$ are null. 

Eventually, to prove that the value of $\al_{s_1, s_2}$ and $\al_{t_1,
  t_2}$ is  $-\beta$, we use  equations~(\ref{eq:A}) and~(\ref{eq:B}) with
$t=t_1$ and $s=s_1$.
\bigskip

\qquad\textit{\underline{Case 3}}: $|S|\geq 2$ and $|U|\geq 1$

For all $s$ in $S$ and all $t,t'$ distinct in $T\backslash\s{t_1}$, we first prove that $\al_{s, t}$, $-\al_{t,t'}$ and
$-\al_{t_1, \ot}$ are all equal to a constant called $\beta$.

Let $s'$ be a vertex in $S\backslash\s{s}$. The transformations represented in Figures~\ref{fig:st3.1}
and~\ref{fig:st3.2} lead to $\al_{s,t} = -\al_{t,t'}$ and
\begin{equation}
\al_{u_1}\mathbb{1}(t<u_1)+\al_{t_1,t}+\beta=\al_{t}\mathbb{1}(t<u_1).\label{eq:st3.1}
\end{equation}

Equation~(\ref{eq:st3.1}) when $t$ is equal to $\ot$ can be simplified
via $\mathcal T(\s{u_1}$ $\s{s', t, \ot}$ $\s{\ot})$ in $\al_{t_1,\ot}=-\beta$.

\begin{minipage}{1.0\linewidth}
  \begin{center}
    \captionsetup{justification=centering, singlelinecheck=false }
    \begin{minipage}{0.44\linewidth}
      \centering\trtikz{$s'\; t_1$}{$t$}{1}{}{$u_1$}{}{1.5}\\
    \end{minipage}
    \begin{minipage}{0.44\linewidth}
      \centering\trtikz{$s'\; s\; t_1 \; t'$}{$t$}{1}{}{$u_1$}{}{1.8}\\
    \end{minipage}
    \vspace{0.2cm}

    \begin{minipage}{0.44\linewidth}
      \captionof{figure}{$\trd{s', t_1, t}{u_1}{t}$}
      \label{fig:st3.1}
    \end{minipage}
    \begin{minipage}{0.44\linewidth}
      \captionof{figure}{$\trd{s', s, t_1, t, t'}{u_1}{t}$}
      \label{fig:st3.2}
    \end{minipage}
  \end{center}
\end{minipage}
We now show that for all triplets $(s,t,u)\in S\times T\backslash \s{t_1,\ot}\times U$ the two higher
vertices  of the  triplet have  the the  same  representative coefficient
equal to $\al_\ot$.

The       transformation       $\trd{s,t}{\ot}{s}$      leads       to
\begin{equation}
\al_{\ot}=\al_{\max\s{s,t}}\label{eq:st3.2}.
\end{equation}
Then       the      transformation
$\trd{s,t,\ot}{u}{t}$                                            yields
$\al_{t}\id{u<t<s}+\al_u\id{t<u}=\al_t\id{s<t<u} +\al_s\id{t<s}$ which
can be  simplified in
\begin{equation}
\al_{\max\s{u,t}}=\al_{\max\s{s,t}}\label{eq:st3.3}\, ,
\end{equation}
using the
fact      that     $\id{u<t<s}-\id{s<t<u}$      is      equal     to
$\id{u<t}-\id{s<t}$.

If  $t$  is  lower  than  $s$  or  $u$,  equations~(\ref{eq:st3.2})  and
(\ref{eq:st3.3})  prove that  the  two higher  vertices  among the  triplet
$(s,t,u)$   have  representative   variables  equal   to  $\al_{\ot}$.
Otherwise,    the   transformation    $\trd{u}{s,t,\ot}{s,t}$   yields
$\al_{\max\s{u,       s}}=\al_{\ot}$       which      leads       with
equation~(\ref{eq:st3.2}) to the same result.

The next  of this  proof consists in  showing that  the representative
coefficient of  any vertex  which is  not $t_1$ is  equal to  zero. This
result  is true if  at least  two vertices  of $(s_1,t_2,u_1)$  are in
$\s{1,2,3}$ (since  the representative  variables of the  two greatest
representatives are equal to  $\al_{\ot}$). It remains to consider the
cases   in  which   only  one   of   these  three   elements  are   in
$\s{1,2,3}$. Let $x$ be this element.
\begin{itemize}
\item If $x=t_2$, $t_3$ is necessarily in $\s{1,2,3}$. The coefficient
  $\al_{\min(s_1, u_1)}$ is equal to $\al_{\ot}$ and
  $\trd{s_2,t_1,t_2}{s_1,\ot,u_1}{t_2}$ gives the result.
\item If  $x=u_1$, $u_2$ is necessarily in  $\s{1,2,3}$ and $\trd{u_1,
    u_2}{t_2}{u_2}$ enable to conclude.
\item If $x=s_1$,  $s_2$ is necessarily in $\s{1,2,3}$.   If the third
  vertex  of this  set is  $t_1$, $\trd{s_2,t_1}{\ot}{s_2}$  gives the
  result.   Otherwise,  $S$  contains  at  least  three  vertices  and
  $(s_1,s_2,s_3)=(1,2,3)$.      We     conclude    using     $\mathcal
  T(C\cup\s{s_1,s_3,t_1,t_3}\s{\ot}\s{s_3})$   with  $C$  successively
  equal  to  $\emptyset$  and  $\s{s_2,t_2}$ and  the  transformation
  $\trd{s_1,s_2,s_3,t_1,t_3,\ot}{t_2}{s_1,s_2,\ot}$.
\end{itemize}


Let $s$ be either $s_1$ or $s_2$. If $t_1\in\s{4,\hdots, n}$, we prove that $\al_{t_1}$ is null
thanks to $\trd{s,t_1,u_1}{\ot}{t_1,\ot}$ and $\trd{s_1, s_2,t_1, t_2,u_1}{\ot}{t_1,\ot}$.
Eventually, we prove for all $s,s'\in S$ distincts and
$t\in T\backslash\s{t_1}$ that expressions $\al_{s, t_1}$,
$-\al_{s, s'}$ and $-\al_{t_1, t}$ are equal to $\beta$ through:   $\trd{s,   t_1}{\ot}{s}$,   $\trd{s,   s',   t_1,
  t_2}{\ot}{s}$ and $\trd{s_1, t_1, t}{s_2, \ot}{t}$.
\end{proof}

\section{General clique inequalities}
\label{sec:general}

The \textit{clique  inequalities} have  been introduced by  Chopra and
Rao~\cite{chopra1993partition} and correspond to the fact that
for any  $m$-partition $\pi$ ($m\leq K$)  and any set  $Z\subset V$ of
size $K+1$, at  least two vertices of $Z$ are  necessarily in the same
cluster. The clique inequality induced by a given set $Z$ is:
\begin{equation}
  \label{eq:rzsge,kl}
  x(E(Z))\geq 1.
\end{equation}

The general clique inequalities are obtained by increasing the size of
both $Z$ and  the right-hand
side   of  equation~\ref{eq:rzsge,kl}.   Thus,   the  general   clique
inequalities induced by a set $Z\subset V$  of size
$qK + r$ with $q\in\mathbb N$ and $r\in\s{0,1,\hdots, K-1}$ is defined
by Chopra and Rao as:
\begin{equation}
  \label{eq:Z}
  x(E(Z))\geq 
\left(
    \begin{array}{@{}c@{}}
      q+1\\2
    \end{array}
\right) r + 
\left(
    \begin{array}{@{}c@{}}
      q\\2
    \end{array}
\right) (K-r).
\end{equation}

Let $P_Z$ be the face of $P_{n,K}$ defined by equation~(\ref{eq:Z}). As
represented    Figure~\ref{fig:dependent},   the   lower    bound   of
inequality~\eqref{eq:Z} corresponds to the  minimal value of $x(E(Z))$ in
the incidence vector of  a $K$-partition~-- thus ensuring the validity
of this inequality.  It is obtained by setting $q+1$ vertices of $Z$ in each
of  the $r$  first clusters  and  $q$ vertices  in each  of the  $K-r$
remaining clusters.

\begin{figure}[h]
  \centering
  
  \centering{\input{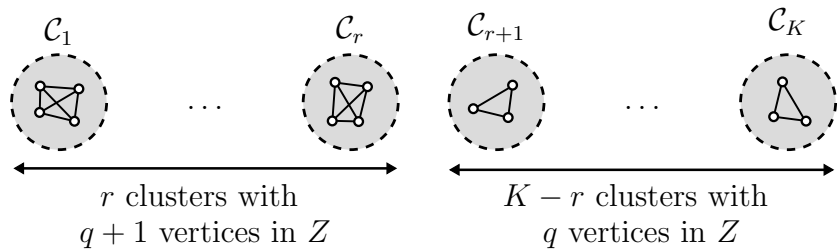}}
  \caption{Distribution of $Z$ vertices in a  $K$-partition
   included in $F_Z$ (case where $q$ is equal to three).}
  \label{fig:dependent}
\end{figure}

These   inequalities   have  also   been   studied   by  Labb\'e   and
\"Oszoy~\cite{labbe2010size}  in  the  case  where the  clusters  must
contain at least $F_L$ vertices. In this context, the size of $Z$ must
greater than or equal to $\lfloor\frac{n}{F_L}\rfloor$.  Finally, Ji and
Mitchell   also   studied   these   inequalities  that   they   called
\textit{pigeon inequalities}~\cite{ji2005clique}.

In the following  $U=\s{u_1, u_2, \hdots, u_{|U|}}$ is  used to denote
$V\backslash  Z$ such that  $u_1<u_2<\hdots<u_{|U|}$. The  vertices in
$Z=\s{z_1, \hdots, z_{K+1}}$ are similarly sorted.

\begin{theorem}
  If $P_{n,K}$ is full-dimensional, for a given $Z\subset V$ of
  size $K+1$, inequality~\eqref{eq:Z} is facet-defining if
  and only if:
  \begin{enumerate}[(i)]
  \item $|U|\geq 1$ and $u_1\leq 3$;
  \item $z_{|Z|}=n$
  \item $|Z|\in\s{K+1,\hdots, 2K-1}$.
  \end{enumerate}
  \label{th:dependent}
\end{theorem}

\begin{lemma}Let $V_1$ and $V_2$ be two disjoint subsets of
  $V$ and let $Z$ be a subset of $V$ which satisfies the conditions of
  Theorem~\ref{th:dependent}. Then, there exists
a $K-$partition in $F_Z$ which includes $V_1$ and $V_2$ if $\s{|V_1\cap
  Z|,|V_2\cap Z|}$ is equal to $\s{1,2}$.
\label{le:val_dep}
\end{lemma}

\begin{proof}
Given the bounds  on the size of $Z$,  $q$ is necessarily
  equal  to  one  and  $r\in\s{1,\hdots,  K-1}$.   Consequently,  each
  $K$-partition included  in $F_Z$ contains at least  one cluster with
  exactly one  vertex in $Z$ and  at least one cluster  with exactly two
  vertices in $Z$.

  Let $\pi=\s{C_1,\hdots, C_K}$ be the $K$-partition such that:
  \begin{itemize}
  \item $C_1=V_1$ and $C_2=V_2$.
  \item Clusters $C_3$ to $C_{r+1}$ each contains $q+1$ vertices from $Z$.
  \item Clusters $C_{r+2}$ to $C_K$ each contains $q$ vertices from $Z$.
  \item The vertices  in $U$ which are not included  in $V_1$ or $V_2$
    are in $C_K$.
  \end{itemize}

This construction is always possible since $|Z|$ is equal to
$qK+r$. It can easily be checked~-- by computing $x^\pi(Z)$~-- that $\pi$
is in $F_Z$.
\end{proof}

Each   transformation   $\mathcal  T(C_1,C_2,R)\mapsto  \s{C'_1,C'_2}$,
considered in the proof of Theorem~\ref{th:dependent} is such that the
couples $(C_1,C_2)$ and $(C'_1,C'_2)$ satisfy the conditions
imposed on $V_1$ and $V_2$ in
Lemma~\ref{le:val_dep}.    This   ensure    the   validity    of   the
transformations. We now present the proof of Theorem~\ref{th:dependent}.

\begin{proof}

  If the first condition of the theorem is not satisfied, the three first vertices are
  in $Z$  and cannot  be in the  same cluster.  Consequently  $P_Z$ is
  included    in    the    hyperplane   defined    by    $\sum_{i=4}^n
  x_i-x_{1,2}-x_{1,3}-x_{2,3}=K-3$.  If  (ii) is false,  the vertices $u$
  which are greater than $z_{K+1}$ cannot be representative since each
  cluster  contains  at least  one  element  of  $Z$. Thus,  $P_Z$  is
  included    in   the   hyperplanes:    $x_u=0\;\forall   u>z_{|Z|}$.
  Eventually, if $Z$ contains  more than $2K-1$ vertices, each cluster
  necessarily include at least  two vertices from $Z$. Thus, $z_{|Z|}$
  cannot be a  representative and $F_Z$ is included  in the hyperplane
  induced by $x_{z_{|Z|}}=0$.



Let  $H=\{ x\in\mathbb  R^{|E|+|V|-3}$ $|  \;\al^T x  = \al_0\}$  be a
hyperplane which includes $P_Z$ and let $z_i<z_j<z_k$ be three elements
of     $Z$.    The     transformation     $\mathcal    T(\s{z_i,z_k},$
$\s{z_j},\s{z_k})$, first shows: $\al_{z_i,z_k}=\al_{z_j, z_k}$.
Thus, for a given $k$ and for all $j\in\s{1,\hdots, k-1}$ the coefficients $\al_{z_j,z_k}$ are equal to a constant,
referred to as $\beta_k$. 

  For all $j$ and $k$ greater than $i$, $\trd{z_i, z_k}{z_j}{z_i}$ gives 
  \begin{equation}
    \label{eq:z1}
    \beta_k-z_k=\beta_j-z_j.
  \end{equation}

Let  $z,z'$  and  $z''$  be   three  distincts  elements  of  $Z$.  The
transformation $\trd{u_1,z}{z'}{u_1}$ leads to
$\al_{z'}+\al_{u_1,z}=\al_{z}+\al_{u_1,  z'}$.  This result   and
$\mathcal T(\s{u_1,z},$ $\s{z',z''},\s{u_1})$ give for all $h\in\s{2,\hdots,K+1}$: $\al_{u_1,z_h}=0\,$ and
\begin{equation}
\al_{u_1,z_1} + \al_{z_h}=\al_{z_1}.\label{eq:z3}
\end{equation}

From     equations~(\ref{eq:z1})      and~(\ref{eq:z3}),     we     obtain
that for  all $h\in\s{2,\hdots,K+1}$, the  representative coefficients
of $z_h$ are equal and that the same applies to the $\beta_h$.

If $|U|$ is equal to one, the proof is over. Indeed, in that case
$z_2$ is lower  than four and thus $\al_{z_2}$ is equal  to zero, which gives via
equation~(\ref{eq:z3}) $\al_{u_1, z_1}=0$.

If $|U|$ is greater than two,  we then prove that $\al_{u,z}$ is equal
to zero  for all $u\in U\backslash \s{u_1}$  and all $z\in Z$.  This is obtained
thanks to $\mathcal T(\s{u_1, u, z_1},\s{z},$ $\s{u_1, u})$, $\trd{u_1, u, z_1}{z,
  z'}{u_1, u}$ and equation~(\ref{eq:z3}).

We  show   that  $\al_{z_1}$  is  equal  to   $\al_{z_2}$,  thanks  to
$\trd{u_2, z_2}{z_1}{u_2}$ which leads through equation~(\ref{eq:z3}) to
$\al_{u_1, z_1}=0$.

If $|U|$ is equal to two, $\al_z$  is equal to zero, and it remains to
prove that  $\al_{u_1, u_2}$ is  equal to zero,  which can be  done by
$\trd{u_1, u_2, z_2}{z_1}{u_2}$.

Otherwise, for a given $u$ in $U$, let $U'$ be a subset of $U\backslash \s{u}$ which contains $u_1$ or
$u_2$. The transformation $\trd{u, z, U'}{n}{u}$ gives 
\begin{equation}
\zsum_{u'\in
  U'} \al_{u,u'} +\al_z =\al_u\,\mbox{ }\forall z\in Z.\label{eq:z7}
\end{equation}
This equation shows that the sum of $\al_{u,u'}$ is 
equal to a  constant for any possible $U'$. Let $u'$  and $u''$ be two
vertices in $U\backslash\s{u}$. By successively choosing
$U'$ equal to $\s{u''}$
and $\s{u',u''}$ we obtain: $\al_{u,u'}=0$.

Eventually,   equation~(\ref{eq:z7})   gives:  $\al_z=\al_u\,$ $  \forall
(u,z)\in  U\times Z$.   Since  $\al_{u_1}$  is equal  to  zero, the  same
applies to the other representative variables.

\end{proof}

\section{Strengthened triangle inequalities}
\label{sec:strength}

Theorem~\ref{th:triangle} states that inequalities~(1)
are  not  facet-defining  if  $s$  is greater  than  both  $t_1$  and
$t_2$. However, they can be strengthened by adding the term $x_s$ to the
left side  of the inequality whenever $s$ is greater  than three (otherwise
$x_s$ is equal to zero):

\begin{equation}
  \tag{2'}\label{eq:reinforcement}
  x_{s,t_1}+x_{s,t_2}-x_{t_1,t_2} + x_s\leq 1.
\end{equation}

For three distinct vertices $s$,  $t_1$ and $t_2$, let $P_{s,t_1,t_2}$ be
the face of  $P_{n,K}$ defined by equation~(\ref{eq:reinforcement}). 

\begin{theorem}
  Let $s,t_1$  and $t_2$ be three  vertices in $V$  such that $s>t_2>t_1$
  and $s>3$. When $P_{n,K}$ is full-dimensional the inequality
  $x_{s,t_1}+x_{s,t_2}-x_{t_1,t_2}+x_s\leq 1$ is facet-defining if
  and only if $(t_2>3)$ or $(K\leq n-3)$.
  \label{th:triangle_reinforcement}
\end{theorem}

\begin{proof}
  Assume that
  $t_2\leq 3$ and $K=n-2$. Let  $\pi$ be a $K-$partition such that the
  three first  vertices are in the  same cluster.  As $K$  is equal to
  $n-2$ the $K-1$  other clusters are necessarily reduced  to one vertex.
  Hence the  left part of equation~(\ref{eq:reinforcement})  is equal to
  zero and $\pi$ is not in $P_{s,t_1,t_2}$.  Since $1$,
  $2$ and $3$ cannot be together we deduce that
  $P_{s,t_1,t_2}$ is included in the hyperplane defined by
  $\sum_{i=4}^n x_i -x_{1,2}-x_{1,3}-x_{2,3}=K-3$. 

  In the following of the proof the terms $t$ and $t'$ refer to either
  $t_1$
  or $t_2$ with the restriction that $t$ is different from $t'$.

Let $\al^Tx = \al_0$ induce an hyperplan which includes  $P_{s,t_1,t_2}$ for $t_2>3$ or $K\leq
n-3$ and let $U=\s{u_1, u_2, \hdots, u_{|U|}}$ be $V\backslash \s{s,t_1, t_2}$ such
  that   $u_1<u_2<  \hdots<u_{|U|}$.   Similarly  to   the   proof  of
  Lemma~\ref{le:uv} we show that  $\al_{u,v}$ is equal to zero for
  all $u\in U$ and all $v\in V\backslash {\s{u}}$.

  Then we  consider the transformation $\trd{C}{t, u_1}{u_1}$ with $C=\s{t', u_2}$ if
  $K=n-2$ and $C=\s{t', u_2, s}$ otherwise. By noting that $\min\s{t',
    u_2}$ is lower than four, we deduce:  $\al_t=0$.
 
  We  then  prove that  $\al_s$  is  equal  to $\al_{s,t}$  thanks  to
  $\trd{t,u_1}{s}{t}$. The    transformation   $\trd{t_1,    u_1,    u}{s}{t_1,u_1}$   gives:
$\al_u=0,\,\forall u\in U\backslash \s{u_1}$. Eventually we obtain, through $\trd{s,t,t'}{u}{t}$: $\al_{t,t'}=-\al_{s,t}$.
\end{proof}

\section{Paw inequalities}
\label{sec:paw}

 Given a subset $W=\s{a,b,c,d}$ of $V$, we define the \textit{paw
inequality} associated to $W$ by:
\begin{equation}
  \label{eq:paw}
  x_{a,b} + x_{b,c} - x_{a,c} + x_{c,d}+ x_b + x_c  \leq 2.
\end{equation}

Figure~\ref{fig:paw} represents the variables  in this
inequality. 

\begin{figure}[h]
  \centering
  \input{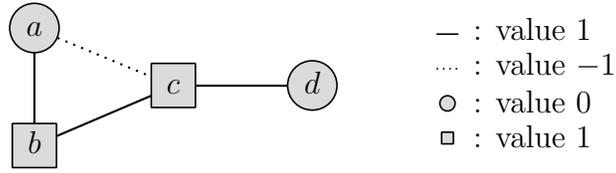}
  \caption{Representation of the coefficients of the paw inequality associated to a subset $\s{a,b,c,d}$ of $V$.}
  \label{fig:paw}
\end{figure}

\begin{lemma}
  Let $K\in\{3,\hdots, n-2\}$. Inequality~\eqref{eq:paw} is valid for
  $P_{n,K}$ if and only if 
  \begin{enumerate}
  \item $a<b$;
  \item $min(b,c,d)=d$.
  \end{enumerate}
\label{le:paw1}
\end{lemma}

\begin{proof}

If $min(b,c,d)$ is not $d$, the left-hand side of
equation~\eqref{eq:paw} is  equal to  $3$ for any  $K$-partition with a
cluster equal to $\s{b,c,d}$. If $b$ is lower than $a$, then
equation~\eqref{eq:paw} is not satisfied  for any $K$-partition
$\pi=\s{C_1,\hdots, C_K}$ such that $\s{a,b}\subset C_1$ and $C_2=\s{c}$.

The addition of the triangle inequality~$(1)$
\begin{equation}
x_{a,b}+x_{b,c}-x_{ac}\leq   1\label{eq:paw1}
\end{equation}
and   the   lower   representative
inequality~\eqref{eq:lrep}
\begin{equation}
  \label{eq:paw2}
  x_c + x_{c,d}\leq 1
\end{equation}
ensures that the  paw inequality is valid if
$x_b$ is equal to zero. 

If $x_b$ is  equal to one, we show  that~\eqref{eq:paw} is still valid
since  equation~\eqref{eq:paw1}  and  equation~\eqref{eq:paw2}  cannot
both  be tight.   In that  case, $a$  and $d$  cannot be  in  the same
cluster  than  $b$  since  their  indices are  lower.   The  only  way
for~\eqref{eq:paw1} to be tight under  these conditions is for $b$ and
$c$  to be  together.   Equation~\eqref{eq:paw2} is  tight  if $c$  is
representative or  if $c$  and $d$ are  together. In both  cases $x_b$
cannot be equal to one if $b$ and $c$ are together.
\end{proof}

Let $F_P$ be the face of $P_{n,K}$ associated to
inequality~\eqref{eq:paw}. 

\begin{lemma}
Under the  conditions of Lemma~\ref{le:paw1}  the face $F_P$ is  not a
facet if $c<b$ or $K=n-2$.
\label{le:paw2}
\end{lemma}

\begin{proof}
  If $c$  is lower  than $b$ we  prove that  $F_P$ is included  in the
  hyperplane induced by  $x_c+x_{c,d}=1$. The expression
  \begin{equation}
x_c+x_{c,d}\label{eq:expre}
\end{equation}
  can be equal to $0$, $1$ or $2$.

  If expression~\eqref{eq:expre} is equal to $0$, the solutions in $F_P$ satisfy :
  $x_{a,b} + x_{b,c} - x_{a,c} + x_b = 2$. This equation cannot be true since $b$ has to be greater than both $a$ and $c$ according
  to  Lemma~\ref{le:paw1}  $b$  and   the  condition  of  the  current
  lemma.   The  expression~\eqref{eq:expre} cannot  be  equal to  two
  either since $d$ is lower than $c$.

  As a result expression~\eqref{eq:expre} is necessarily equal to one.

  If $K$ is equal to $n-2$, no $K$-partition  can contain both
  $a$ and $c$ and thus, $F_P$ is included in the hyperplane induced by $x_{a,c}$. 

\end{proof}

\begin{theorem}
Let $K\in\s{3,n-3}$ and $b\in\s{4,\hdots, n}$, $F_P$ is facet defining
of $P_{n,K}$ if and only if
\begin{enumerate}
\item $d<b<c$;
\item $a<b$.
\end{enumerate}
\label{th:paw}
\end{theorem}

\begin{proof}
 A $K$-partition  containing a cluster equal  to $\s{b,c}$ satisfies
 the paw  inequality. Thus, by  setting $C_3$ equal to  $\s{b,c}$, one
 can use the same reasoning  as in the proof of Theorem~\eqref{th:dim}
 to  obtain  the  following   relations  on  the  coefficients  of  an
 equation $\al^Tx=\al_0$ satisfied by all the points in $F_P$:
 \begin{itemize}
 \item $\al_{i,j}  = 0$ $\forall  i\in V\backslash\s{b,c,1,2,3}\forall
   j\in V\backslash\s{b,c,i}$;
 \item $\al_{1,2}=\al_{1,3}=\al_{2,3}\zdef\beta$;
 \item $\al_i=-2\beta$ $\forall i\in V\backslash\s{b,c,1,2,3}$.
 \end{itemize}

The value of the remaining  $\al$ coefficients can be obtained through
the transformations represented table~\ref{tab:paw}.

\end{proof}

\begin{table}[h!]
  \centering
  \begin{tabular}{M{4cm}ccc}
    \toprule
    \textbf{Conditions} & \textbf{Transformation}
    & \textbf{Results} \\\midrule
    
    - & \trtikz{$b$}{$c$}{1}{}{$d$}{}{1.5} &
    $\al_{b,c}=\al_{c,d}\zdef\gamma$\\\midrule

    $\forall          i\in          V\backslash\s{a,b,c,d}$          &
    \trtikz{$d\;i$}{$c$}{1}{}{$b$}{}{1.5} & $\al_{c,i}=0$\\\rule[-7pt]{0pt}{34pt}
    & \trtikz{$a\;b$}{$i$}{1}{$c$}{$d$}{}{1.5} & $\al_{b,i}=0$
    \\\midrule

    if    $d\geq   4$   $\forall    e,f\in   \s{1,2,3}$
    $\s{a,b}\subset C_3$& \trtikz{$c$}{$d$}{}{$e$}{$f$}{1}{1.5} & $\beta = 0$\\
    \midrule

    if $d\leq 3$ $\forall e,f\in \s{1,2,3}\backslash\s{d}$ $C_2\subset
    V\backslash\s{b,c,d,e,f}$    $C_3=\s{b,c,d}$    &
    \trtikz{$e$}{$f$}{1}{}{$C_2$}{}{1.5} & $\beta=0$\\\midrule

    $\forall i\in V\backslash\s{a,b,c,d}$  & \trtikz{$a\; i$}{$b$}{1}{}{$c$}{}{1.5}  & $\al_c  + \al_{a,b}
    =\gamma + \al_b$&\\\rule[-7pt]{0pt}{26pt}

     & \trtikz{$a\; i\;d$}{$b$}{1}{}{$c$}{}{1.5} &
    $\al_{b,d} = 0$ \\\rule[-7pt]{0pt}{26pt}

    &   \trtikz{$c\;d$}{$a\;    b$}{1}{}{$i$}{}{1.5}   &   $\al_{a,c}=-\gamma$\\\rule[-7pt]{0pt}{26pt}

    \rule[-1.5ex]{0pt}{4ex} & \trtikz{$b\;c\;d$}{$a$}{1}{}{$i$}{1}{1.5} & $\al_{a,b} = \gamma,\;
    \al_b=\al_c$\\\rule[-7pt]{0pt}{26pt}

    & \trtikz{$b\;c$}{$a\;d$}{1}{}{$i$}{}{1.5} & $\al_b=\gamma$

    \\\bottomrule
  \end{tabular}
  \caption{Transformations  used  in  Theorem~\ref{th:paw}. Each  line
    presents a step of the proof. The last column corresponds to the result.
    column. }
  \label{tab:paw}
\end{table}

\begin{theorem}
Let $K\in\s{3,n-3}$. The face $F_P$ associated to
the  inequality  $x_{a,b}+x_{b,c}- x_{a,c}  +x_{c,d}+x_c  + x_{1,2}  -
\sum_{i=4}^n\leq    4-K$~--    which    corresponds   to    the    paw
inequality~\eqref{eq:paw} for $b$ equal to three~--  is facet for $P_{n,K}$ if and only if
\begin{enumerate}
\item $d<3<c$;
\item $a<3$.
\end{enumerate}
\label{th:paw2}
\end{theorem}

\clearpage

\begin{proof}
  Assume  that  $F_P$  is   included  in  the  hyperplane  induced  by
  $\al^Tx=\al_0$.  Let   $i$  and  $j$   be  two  distinct   nodes  in
  $V\backslash{\s{a,b,c,d}}$. Lemma~\ref{lemma1} can be used with $\s{1,i}\subset
  C_1$, $\s{2,j}\subset C_2$ and $C_3=\s{b,c}$ to prove that $\al_{i,j}=0$.  

  To show  that $\al_{2,i}$ is equal  to zero we apply  the same lemma
  with $\s{a,d}\subset C_1$, $\s{b,i}\subset C_2$. To ensure that
  the $K$-partition considered are in  $F_P$, $C_3$ is equal to $c$ if
  $d=1$ and $\s{c}$ is in $C_2$ otherwise.

  We then use Lemma~\ref{lemma2} with $\s{1,b}\subset C_1$ and
  $\s{2,i}\subset C_2$ to show that $\al_{1,i}$ is equal to zero. This
  time  the  validity of  the  $K$-partitions  is  ensured by  setting
  $C_3=\s{c}$ if $d$ is equal to two and $c\subset C_1$ otherwise.

  The  value  of the  remaining  $\al$  coefficients  can be  obtained
  through the transformations represented table~\ref{tab:paw3}.

\begin{table}[h!]
  \centering
  \begin{tabular}{M{4cm}ccc}
    \toprule
    \textbf{Conditions} & \textbf{Transformation}
    & \textbf{Results} \\\midrule

  $\s{b,c,d}\subset C_3$ &
  \trtikz{$i$}{$a$}{1}{}{$j$}{}{1.5}      &      $\al_i=\al_j\zdef\beta$     &
  \\\midrule
  
  $\s{b,c}\subset C_3$ & \trtikz{$a$}{$i$}{}{$j$}{$d$}{1}{1.5} & $\al_{a,d} = -\beta$\\\midrule

  $\s{c,d}\subset   C_3$  &   \trtikz{$a$}{$b$}{}{$j$}{$i$}{1}{1.5}  &
  $\al_{b,i}=0$\\\rule[-7pt]{0pt}{22pt}
  
  & \trtikz{$a$}{$b$}{1}{}{$i$}{}{1.5} & $\al_{a,b} = -\beta$\\\midrule
  
    - & \trtikz{$b$}{$c$}{1}{}{$d$}{}{1.5} &
    $\al_{b,c}=\al_{c,d}\zdef\gamma$\\\rule[-7pt]{0pt}{22pt}

              &
    \trtikz{$d\;i$}{$c$}{1}{}{$b$}{}{1.5} & $\al_{c,i}=0$ \\\rule[-7pt]{0pt}{22pt}

        & \trtikz{$a\; i$}{$b$}{1}{}{$c$}{}{1.5}  & $\al_c -\beta
    =\gamma$& \\\rule[-7pt]{0pt}{26pt}

     & \trtikz{$a\; i\;d$}{$b$}{1}{}{$c$}{}{1.5} &
    $\al_{b,d} = 0$ \\\rule[-7pt]{0pt}{26pt}

    &        \trtikz{$c\;d$}{$a\;        b$}{1}{}{$i$}{}{1.5}        &
    $\al_{a,c}=-\gamma$\\\rule[-7pt]{0pt}{26pt}

    & \trtikz{$b\;c\;d$}{$a$}{1}{}{$i$}{1}{1.5} & $\gamma = -
    \beta,\;\al_c=0$\\

 \bottomrule
  \end{tabular}
  \caption{Transformations  used  in  Theorem~\ref{th:paw2}. Each  line
    presents a step of the proof. The last column corresponds to the result.}
  \label{tab:paw3}
\end{table}

\end{proof}

\section{Numerical experiments}
\label{sec:num}


 

In this  section we study the  strength of our formulation  and of the
reinforcements with facets of the previous sections. We consider three
data sets generated randomly, and we believe the instances to be quite
difficult as  there are  no preexisting classes  to detect.  Data sets
$D_1$, $D_2$, $D_3$ all contain  100 instances formed from complete
graphs. In  $D_1$, $D_2$  and $D_3$,  the edge
weights are respectively in $[0, 500]$, $[-250,250]$ and $[-500,0]$.

We  first compare  the  value  of the  linear  relaxation from  our
formulation  and   the  one  from   the  formulation  of   Chopra  and
Rao~\cite{chopra1993partition}                                    (also
in~\cite{kaibel2011orbitopal,fan2010linear}).  The  results  over  the
three  data sets  are  displayed in  tables~\ref{tab:D1},~\ref{tab:D2}
and~\ref{tab:D3}. In each table and  for each couple $(n,K)$ the value
corresponding  to   formulation  $(P_2)$   is  the  second   one.  Our
formulation gives better relaxation values in all cases except when is
equal to $2$.

We now only focus on formulation $(P_2)$. Tables~\ref{tab:isRoot} and \ref{tab:isRoot2} show the number of
instances whose  linear relaxation  gives an optimal  integer solution
for data set $D_1$ and $D_2$.  No optimal solution is obtained for the
instances  of  $D_3$.  Data  set  $D_1$  has  the  hardest  instances  in
practice.   $D_2$  instances   have   weights  of   both  signs,   like
in~\cite{grotschel1989cutting},  and $D_3$  corresponds  to a  variant
which it  considered to be easier  (minimizing a cut with  $K$ parts and
positive weights).

To  evaluate the  efficiency of  a family  of inequalities,  we  use a
separation  algorithm to  add  some  of them  to  the formulation  and
observe  the percentage  of improvement  of  the value  of the  linear
relaxation. This  necessitates the definition  of separation algorithms
for each of the family considered.

\begin{table}[h!]\renewcommand{\arraystretch}{1.2}\centering \begin{tabular}{M{0.5cm}*{19}{r@{\hspace{0.5cm}}}r@{}}\toprule\multirow{2}{*}{\textbf{n}} & \multicolumn{8}{c}{\textbf{K}} \\& \textbf{2} & \textbf{3} &\textbf{4} &\textbf{5} &\textbf{6} &\textbf{7} &\textbf{8} &\textbf{9} &\textbf{10} \tabularnewline\hline
\multirow{2}{*}{\textbf{10}} 	&	1102&	114&	40&	16&	7&	3&	2&	1&	0\\
 &	978&	679&	462&	295&	172&	88&	39&	11&	0\\\hline
\multirow{2}{*}{\textbf{11}} 	&	1268&	123&	45&	19&	8&	4&	2&	1&	0\\
 &	1033&	737&	522&	356&	227&	135&	71&	33&	10\\\hline
\multirow{2}{*}{\textbf{12}} 	&	1408&	140&	49&	20&	9&	4&	2&	1&	0\\
 &	1124&	816&	593&	422&	288&	187&	112&	59&	27\\\hline
\multirow{2}{*}{\textbf{13}} 	&	1529&	125&	46&	20&	8&	4&	2&	1&	0\\
 &	1177&	874&	649&	476&	341&	231&	149&	90&	47\\\hline
\multirow{2}{*}{\textbf{14}} 	&	1607&	124&	46&	17&	8&	4&	2&	1&	0\\
 &	1207&	904&	687&	521&	387&	278&	191&	123&	72\\\hline
\multirow{2}{*}{\textbf{15}} 	&	1733&	125&	48&	20&	8&	4&	2&	1&	0\\
 &	1275&	971&	749&	578&	440&	327&	234&	159&	100\\\hline
\multirow{2}{*}{\textbf{16}} 	&	1883&	127&	44&	19&	8&	4&	2&	1&	0\\
 &	1284&	993&	784&	623&	490&	378&	284&	206&	142\\\hline
\multirow{2}{*}{\textbf{17}} 	&	2010&	118&	43&	19&	8&	4&	2&	1&	0\\
 &	1375&	1079&	858&	684&	542&	426&	329&	246&	177\\\hline
\multirow{2}{*}{\textbf{18}} 	&	2055&	117&	41&	18&	9&	4&	2&	1&	0\\
 &	1391&	1105&	893&	726&	589&	474&	377&	293&	221\\\hline
\multirow{2}{*}{\textbf{19}} 	&	2270&	133&	49&	21&	9&	4&	2&	1&	1\\
 &	1463&	1164&	941&	765&	621&	503&	404&	317&	243\\\hline
\multirow{2}{*}{\textbf{20}} 	&	2359&	123&	43&	17&	8&	4&	2&	1&	0\\
 &	1439&	1146&	936&	774&	642&	530&	433&	348&	273\\
\bottomrule
\end{tabular}
\caption{Mean value of the  linear relaxation from formulation $(P_2)$
  and     the      formulation     proposed     by      Chopra     and
  Rao~\cite{chopra1993partition}  over the  data set  $D_1$.  For each
  couple  $(n,K)$   the  value  on  the  second   line  corresponds  to
  formulation $(P_2)$.}
\label{tab:D1}\end{table}

\clearpage
\begin{table}[h!]\renewcommand{\arraystretch}{1.2}\centering \begin{tabular}{M{0.5cm}*{19}{r@{\hspace{0.2cm}}}r@{}}\toprule\multirow{2}{*}{\textbf{n}} & \multicolumn{8}{c}{\textbf{K}} \\& \textbf{2} & \textbf{3} &\textbf{4} &\textbf{5} &\textbf{6} &\textbf{7} &\textbf{8} &\textbf{9} &\textbf{10} \tabularnewline\hline
\multirow{2}{*}{\textbf{10}} 	&	-2246&	-2251&	-2213&	-2146&	-2064&	-1940&	-1754&	-1409&	0\\
 &	-1566&	-1620&	-1570&	-1431&	-1236&	-990&	-709&	-385&	0\\\hline
\multirow{2}{*}{\textbf{11}} 	&	-2784&	-2778&	-2740&	-2672&	-2591&	-2484&	-2341&	-2136&	-1731\\
 &	-1861&	-1919&	-1869&	-1746&	-1563&	-1330&	-1062&	-752&	-407\\\hline
\multirow{2}{*}{\textbf{12}} 	&	-3374&	-3352&	-3314&	-3245&	-3159&	-3057&	-2931&	-2775&	-2532\\
 &	-2115&	-2171&	-2149&	-2050&	-1891&	-1679&	-1422&	-1131&	-804\\\hline
\multirow{2}{*}{\textbf{13}} 	&	-4159&	-4140&	-4091&	-4017&	-3933&	-3833&	-3717&	-3577&	-3389\\
 &	-2557&	-2606&	-2576&	-2480&	-2326&	-2119&	-1867&	-1575&	-1250\\\hline
\multirow{2}{*}{\textbf{14}} 	&	-4936&	-4901&	-4846&	-4768&	-4674&	-4566&	-4452&	-4318&	-4157\\
 &	-2938&	-2997&	-2978&	-2891&	-2747&	-2556&	-2318&	-2037&	-1719\\\hline
\multirow{2}{*}{\textbf{15}} 	&	-5732&	-5693&	-5634&	-5552&	-5467&	-5361&	-5243&	-5115&	-4965\\
 &	-3332&	-3399&	-3388&	-3314&	-3180&	-2994&	-2759&	-2483&	-2171\\\hline
\multirow{2}{*}{\textbf{16}} 	&	-6683&	-6641&	-6582&	-6496&	-6410&	-6306&	-6193&	-6067&	-5919\\
 &	-3803&	-3861&	-3850&	-3779&	-3648&	-3467&	-3242&	-2970&	-2662\\\hline
\multirow{2}{*}{\textbf{17}} 	&	-7510&	-7484&	-7428&	-7347&	-7257&	-7150&	-7038&	-6919&	-6783\\
 &	-4250&	-4308&	-4304&	-4240&	-4117&	-3945&	-3726&	-3463&	-3161\\\hline
\multirow{2}{*}{\textbf{18}} 	&	-8539&	-8510&	-8449&	-8364&	-8272&	-8163&	-8050&	-7925&	-7784\\
 &	-4788&	-4839&	-4829&	-4768&	-4657&	-4492&	-4277&	-4014&	-3711\\\hline
\multirow{2}{*}{\textbf{19}} 	&	-9606&	-9559&	-9501&	-9412&	-9319&	-9209&	-9093&	-8960&	-8828\\
 &	-5300&	-5361&	-5357&	-5306&	-5199&	-5041&	-4839&	-4588&	-4300\\\hline
\multirow{2}{*}{\textbf{20}} 	&	-10770&	-10725&	-10666&	-10576&	-10483&	-10374&	-10259&	-10135&	-10000\\
 &	-5936&	-5991&	-5979&	-5915&	-5797&	-5628&	-5410&	-5153&	-4860\\
\bottomrule
\end{tabular}
\caption{Mean value of the  linear relaxation from formulation $(P_2)$
  and     the      formulation     proposed     by      Chopra     and
  Rao~\cite{chopra1993partition}  over the  data set  $D_2$.  For each
  couple $(n,K)$ the value on the second line corresponds to
  formulation $(P_2)$.}
\label{tab:D2}\end{table}
\begin{table}[h!]\renewcommand{\arraystretch}{1.2}\centering \begin{tabular}{@{}M{0.5cm}@{\hspace{0.2cm}}*{19}{r@{\hspace{0.25cm}}}r@{}}\toprule\multirow{2}{*}{\textbf{n}} & \multicolumn{8}{c}{\textbf{K}} \\& \textbf{2} & \textbf{3} &\textbf{4} &\textbf{5} &\textbf{6} &\textbf{7} &\textbf{8} &\textbf{9} &\textbf{10} \tabularnewline\hline
\multirow{2}{*}{\textbf{10}} 	&	-10929&	-10438&	-9935&	-9294&	-8622&	-7848&	-6925&	-5549&	0\\
 &	-9959&	-8724&	-7488&	-6253&	-5017&	-3780&	-2540&	-1293&	0\\\hline
\multirow{2}{*}{\textbf{11}} 	&	-13490&	-13001&	-12498&	-11865&	-11191&	-10453&	-9589&	-8557&	-6986\\
 &	-12375&	-11009&	-9643&	-8277&	-6911&	-5545&	-4174&	-2801&	-1422\\\hline
\multirow{2}{*}{\textbf{12}} 	&	-16107&	-15611&	-15111&	-14477&	-13802&	-13068&	-12246&	-11317&	-10185\\
 &	-14876&	-13397&	-11918&	-10439&	-8960&	-7480&	-5999&	-4517&	-3028\\\hline
\multirow{2}{*}{\textbf{13}} 	&	-19402&	-18912&	-18414&	-17773&	-17102&	-16372&	-15574&	-14680&	-13651\\
 &	-18019&	-16390&	-14761&	-13131&	-11502&	-9872&	-8243&	-6611&	-4977\\\hline
\multirow{2}{*}{\textbf{14}} 	&	-22638&	-22135&	-21626&	-20988&	-20311&	-19567&	-18784&	-17912&	-16963\\
 &	-21141&	-19390&	-17639&	-15888&	-14137&	-12385&	-10634&	-8883&	-7130\\\hline
\multirow{2}{*}{\textbf{15}} 	&	-26015&	-25517&	-25003&	-24366&	-23716&	-22983&	-22206&	-21357&	-20428\\
 &	-24402&	-22537&	-20673&	-18808&	-16943&	-15078&	-13214&	-11347&	-9481\\\hline
\multirow{2}{*}{\textbf{16}} 	&	-29871&	-29372&	-28864&	-28226&	-27573&	-26841&	-26078&	-25250&	-24336\\
 &	-28121&	-26122&	-24124&	-22125&	-20126&	-18128&	-16129&	-14130&	-12132\\\hline
\multirow{2}{*}{\textbf{17}} 	&	-33803&	-33308&	-32805&	-32174&	-31515&	-30786&	-30034&	-29226&	-28346\\
 &	-31935&	-29818&	-27702&	-25585&	-23469&	-21352&	-19236&	-17118&	-15001\\\hline
\multirow{2}{*}{\textbf{18}} 	&	-38126&	-37624&	-37113&	-36473&	-35816&	-35081&	-34329&	-33520&	-32650\\
 &	-36139&	-33893&	-31648&	-29402&	-27157&	-24911&	-22666&	-20420&	-18174\\\hline
\multirow{2}{*}{\textbf{19}} 	&	-42640&	-42143&	-41643&	-41001&	-40352&	-39621&	-38871&	-38047&	-37207\\
 &	-40515&	-38143&	-35771&	-33399&	-31027&	-28655&	-26283&	-23911&	-21539\\\hline
\multirow{2}{*}{\textbf{20}} 	&	-47565&	-47065&	-46560&	-45916&	-45262&	-44534&	-43788&	-42981&	-42148\\
 &	-45309&	-42801&	-40294&	-37786&	-35279&	-32771&	-30263&	-27755&	-25247\\
\bottomrule
\end{tabular}
\caption{Mean value of the  linear relaxation from formulation $(P_2)$
  and     the      formulation     proposed     by      Chopra     and
  Rao~\cite{chopra1993partition}  over the  data set  $D_3$.  For each
  couple $(n,K)$ the value on the second line corresponds to
  formulation $(P_2)$.}
\label{tab:D3}\end{table}

\clearpage



\begin{table}[h!]\renewcommand{\arraystretch}{1.2}\centering \begin{tabular}{c@{\hspace{0.2cm}}*{19}{r@{\hspace{0.2cm}}}}\toprule\multirow{2}{*}{\textbf{n}}
    &  \multicolumn{17}{c}{\textbf{K}}   \\&  \textbf{2}  &  \textbf{3}
    &\textbf{4}   &\textbf{5}   &\textbf{6}  &\textbf{7}   &\textbf{8}
    &\textbf{9} &\textbf{10} &  \textbf{11}& \textbf{12} & \textbf{13}
    &  \textbf{14}  &  \textbf{15}   &  \textbf{16}  &  \textbf{17}  &
    \textbf{18} & \textbf{19}\tabularnewline\hline
\textbf{10} & 0 & 0 & 0 & 12 & 47 & 75 & 95 & 100\\
\textbf{11} & 0 & 0 & 0 & 0 & 15 & 50 & 77 & 94 & 100\\
\textbf{12} & 0 & 0 & 0 & 0 & 8 & 29 & 70 & 91 & 99 & 100\\
\textbf{13} & 0 & 0 & 0 & 0 & 0 & 8 & 37 & 60 & 87 & 99 & 100\\
\textbf{14} & 0 & 0 & 0 & 0 & 0 & 0 & 20 & 45 & 72 & 91 & 98 & 100\\
\textbf{15} & 0 & 0 & 0 & 0 & 0 & 0 & 6 & 28 & 61 & 84 & 94 & 99 & 100\\
\textbf{16} & 0 & 0 & 0 & 0 & 0 & 0 & 3 & 13 & 31 & 66 & 80 & 91 & 98 & 100\\
\textbf{17} & 0 & 0 & 0 & 0 & 0 & 0 & 0 & 3 & 18 & 46 & 69 & 84 & 90 & 96 & 100\\
\textbf{18} & 0 & 0 & 0 & 0 & 0 & 0 & 0 & 0 & 6 & 19 & 51 & 70 & 91 & 97 & 99 & 100\\
\textbf{19} & 0 & 0 & 0 & 0 & 0 & 0 & 0 & 0 & 1 & 13 & 34 & 56 & 76 & 92 & 96 & 99 & 100\\
\textbf{20} & 0 & 0 & 0 & 0 & 0 & 0 & 0 & 0 & 0 & 4 & 12 & 37 & 60 & 78 & 94 & 99 & 100 & 100\\
\bottomrule\end{tabular}
\caption{Number of instances of $D_1$ whose linear relaxation gives an
    optimal solution.}
  \label{tab:isRoot}
\end{table}

\begin{table}[h!]\renewcommand{\arraystretch}{1.2}\centering \begin{tabular}{c@{\hspace{0.2cm}}*{19}{m{0.5cm}@{\hspace{0.19cm}}}}\toprule\multirow{2}{*}{\textbf{n}}
    &  \multicolumn{17}{c}{\textbf{K}}   \\&  \textbf{2}  &  \textbf{3}
    &\textbf{4}   &\textbf{5}   &\textbf{6}  &\textbf{7}   &\textbf{8}
    &\textbf{9} &\textbf{10} &  \textbf{11}& \textbf{12} & \textbf{13}
    &  \textbf{14}  &  \textbf{15}   &  \textbf{16}  &  \textbf{17}  &
    \textbf{18} & \textbf{19}\tabularnewline\hline
\textbf{10} & 20 & 53 & 31 & 4 & 3 & 0 & 0 & 0\\
\textbf{11} & 16 & 37 & 16 & 4 & 1 & 1 & 0 & 0 & 0\\
\textbf{12} & 8 & 21 & 16 & 4 & 0 & 0 & 0 & 0 & 0 & 0\\
\textbf{13} & 12 & 17 & 11 & 2 & 1 & 0 & 0 & 0 & 0 & 0 & 0\\
\textbf{14} & 6 & 14 & 9 & 3 & 0 & 0 & 0& 0 & 0 & 0 & 0 & 0\\
\textbf{15} & 2 & 10 & 9 & 3 & 0 & 0 & 0 & 0 & 0 & 0 & 0 & 0 & 0\\
\textbf{16} & 2 & 5 & 5 & 1 & 0 & 0 & 0 & 0 & 0 & 0 & 0 & 0 & 0 & 0\\
\textbf{17} & 2 & 10 & 4 & 2 & 0 & 0 & 0 & 0 & 0 & 0 & 0 & 0 & 0 & 0 & 0\\
\textbf{18} & 2 & 5 & 0 & 0 & 0 & 0 & 0 & 0 & 0 & 0 & 0 & 0 & 0 & 0 & 0 & 0\\
\textbf{19} & 1 & 3 & 0 & 0 & 0 & 0 & 0 & 0 & 0 & 0 & 0 & 0 & 0 & 0 & 0 & 0 & 0\\
\textbf{20} & 0 & 1 & 0 & 0 & 0 & 0 & 0 & 0 & 0 & 0 & 0 & 0 & 0 & 0 & 0 & 0 & 0 & 0\\
\bottomrule\end{tabular}
\caption{Number of instances of $D_2$ whose linear relaxation gives an
    optimal solution.}
\label{tab:isRoot2}
\end{table}


The  values of $n$  considered in  our experiments  are low  enough to
allow   an    exhaustive   enumeration   of   all    the   valid   paw
inequalities.  Separating the  $2-$partition  inequalities is  NP-hard
\cite{oosten2001clique}  and we are  not able  to enumerate  them all.
Instead we use a  heuristic inspired from the well-known Kernighan-Lin
algorithm \cite{kernighan1970eflicient}.  A  similar procedure is used
for the separation of the general clique inequalities.  \bigskip

Separating  the   $2-$chorded  cycle   inequalities  is  a   bit  more
technical. In \cite{muller1996partial},  M\"uller adapted an approach,
introduced  by  Barahona  and  Mahjoub~\cite{mahjoub1986polytope},  to
separate in  polynomial time odd closed walk  inequalities in directed
graphs.  M\"uller  showed that  the same algorithm  can be  applied to
undirected graphs to  allow the separation of a  class of inequalities
which includes  the $2-$chorded  cycle inequalities.  We  adapted this
approach  to separate $2-$chorded  cycle inequalities  from cycles
which may contain repetitions.
\newpage

We define  a graph $H=(V_H, A_H)$  such that for each  edge $ij\in E$,
$A_H$ contains (see example Figure~\ref{fig:tcc-separation}):
\begin{itemize}
\item  eight  vertices:  $u_1^{ij},\,u_2^{ij},\,v_1^{ij},\,  v_2^{ij},\,
  u_1^{ji},\,u_2^{ji},\,v_1^{ji}$ and $v_2^{ji}$; 
\item  four  arcs:  $(u_1^{ij},  u_2^{ij})$,  $(v_1^{ij},  v_2^{ij})$,
  $(u_1^{ji}, u_2^{ji})$, $(v_1^{ji}, v_2^{ji})$ of weight $x_{ij}$.
\end{itemize}

Moreover, to each pair of edges $ij,ik\in E$ with a common endnode, we
associate  four  additional   arcs  in  $A_H$:  $(u_2^{ji},v_1^{ik})$,
$(v_2^{ji},u_1^{ik})$, $(u_2^{ki},v_1^{ij})$ and $(v_2^{ki},u_1^{ij})$
of weight $-x_{jk}-\frac{1}{2}$.

Let $C=\s{c_1, \hdots, c_{2p+1}}$ be an odd cycle of $G$. By construction,
$C$   induces   a   walk    in   $H$   from   $u_1^{c_1,c_{2}}$   to
$v_1^{c_1,c_{2}}$ (see example Figure~\ref{fig:tcc-cycle}) of weight
\[
\begin{array}{l@{}l@{\hspace{0.1cm}}l}
x_{c_1,c_2}  - \frac{1}{2} -  x_{c_1, c_3} + \hdots  + x_{c_{2p+1},
  c_1} - \frac{1}{2} - x_{c_{2p+1},c_2} \\ \qquad= x(E(C)) - x(E(\overline C)) -
 \frac{2p+1}{2}\\
\qquad = x(E(C)) - x(E(\overline C)) -
\lfloor\frac{|C|}{2}\rfloor - \frac{1}{2}.
\end{array}
\]

Thus, there exists a cycle $C$ which violates inequality (\ref{eq:tc}) if and
only if there exists a  path from $u_1^{c_1,c_2}$ to $v_1^{c_1,c_2}$ in
$H$ whose length is greater than $-\frac{1}{2}$. 

M\"uller's approach  for undirected graphs only considers  four vertices per
edge   ($u_1^{ij},\,u_2^{ij},\,v_1^{ij}$   and   $v_2^{ij}$).   As   a
consequence, a path in $H$ between $u_1^{ij}$ and $v_1^{ij}$
corresponds to a sequence of edges in $G$ such that each edge has a
common endnode with its neighbors. Such a sequence may not be a cycle
(\textit{e.g:}  $\s{ij, ik,  il}$).  Four additional  vertices per  edge
enable to  give an  orientation to the  edge in the  obtained sequence
and thus ensure that it is a cycle (possibly with vertex repetitions).

\begin{figure}[h]
  \centering
  \scalebox{0.9}{\input{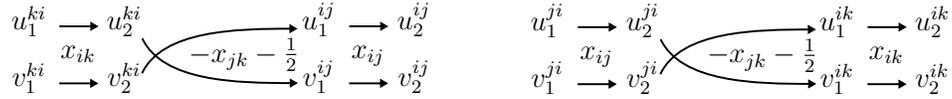}}
  \caption{Vertices and arcs of $H$ associated to edges $(ij)$ and $(ik)$
    in $E$.}
  \label{fig:tcc-separation}
\end{figure}

\begin{figure}[h]
  \centering
  \input{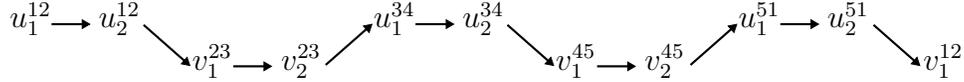}
  \caption{Path    in   $H$   which    corresponds   to    the   cycle
    $C=\s{1,2,3,4,5}$ in $G$.}
  \label{fig:tcc-cycle}
\end{figure}

After  creating $H$, we  obtain for  all $ij\in  E$ the  shortest path
between  $u_1^{ij}$ and $v_1^{ij}$  thanks to  Floyd-Warshall shortest
path algorithm  \cite{ahuja1993network}. And deduce  the corresponding
cycle    in    $G$    and    its    associated    $2-$chorded    cycle
inequality.  Eventually, the  violated inequalities  are added  to the
problem and the  root relaxation is updated. This  process is repeated
until no more violated inequality is found.



\bigskip


For the instances in $D_1$~--  which are the most difficult ones~-- no
violated paw  inequality has been found.  This  family of inequalities
does    not     seem    to     be    efficient    in     this    case.
Table~\ref{tab:tcc_graph1}  shows  the  results obtained  when  adding
2-chorded cycle  inequalities.  The mean percentage  of improvement is
low  (lower than  $5\%$).  These  inequalities are  less likely  to be
efficient  for  this type  of  instances.   Moreover, the  improvement
decreases when  $K$ gets  closer to  $n$.  This is  true also  for the
other classes  of inequalities,  and it can  be explained by  the fact
that the  number of  solutions solved to  optimality increases  in this
part of the tables.

\begin{table}[h!]\renewcommand{\arraystretch}{1.2}\centering \begin{tabular}{M{0.5cm}*{9}{r@{\hspace{0.5cm}}}r@{}}\toprule\multirow{2}{*}{\textbf{n}} & \multicolumn{8}{c}{\textbf{K}} \\& \textbf{2} & \textbf{3} &\textbf{4} &\textbf{5} &\textbf{6} &\textbf{7} &\textbf{8} &\textbf{9} &\textbf{10} \tabularnewline\hline
\textbf{10} & 
1,2 & 1,5 & 0,9 & 0,7 & 0,1 & 0,1 & 0,0 & 0,0 & \\
\textbf{11} & 
0,9 & 1,1 & 0,9 & 0,6 & 0,2 & 0,0 & 0,0 & 0,0 & 0,0 & \\
\textbf{12} & 
1,7 & 1,9 & 1,7 & 1,4 & 1,0 & 0,6 & 0,2 & 0,0 & 0,0 & \\
\textbf{13} & 
1,8 & 2,0 & 1,9 & 1,5 & 0,6 & 0,3 & 0,0 & 0,0 & 0,0 & \\
\textbf{14} & 
2,2 & 2,5 & 2,5 & 2,1 & 1,3 & 0,7 & 0,4 & 0,1 & 0,0 & \\
\textbf{15} & 
2,1 & 2,9 & 3,0 & 2,5 & 1,9 & 1,1 & 0,5 & 0,2 & 0,0 & \\
\textbf{16} & 
2,5 & 2,8 & 2,7 & 2,1 & 1,7 & 1,4 & 0,7 & 0,4 & 0,1 & \\
\textbf{17} & 
2,9 & 3,2 & 3,0 & 2,7 & 2,3 & 1,7 & 1,4 & 0,7 & 0,3 & \\
\textbf{18} & 
2,9 & 3,4 & 3,3 & 2,9 & 2,5 & 1,9 & 1,3 & 0,8 & 0,4 & \\
\textbf{19} & 
3,2 & 3,6 & 3,9 & 3,7 & 3,2 & 2,6 & 2,0 & 1,2 & 0,7 & \\
\textbf{20} & 
4,1 & 4,7 & 4,8 & 4,6 & 4,1 & 3,4 & 3,0 & 2,6 & 1,6 & \\
\bottomrule
\end{tabular}
\caption{
  Mean gain in percentage over the instances of $D_1$ when adding 
  2-chorded cycle inequalities.
}
\label{tab:tcc_graph1}
\end{table}

Table~\ref{tab:st_graph1}       and~\ref{tab:dep_graph1}       present
respectively the results obtained with the 2-partition and the general
clique inequalities. The improvement  is significantly higher than for
the  2-chorded cycle  inequalities.  The  general  clique inequalities
lead  to a  spectacular  improvement  of the  solution  of the  linear
relaxation  for the  lowest values  of $K$.

\begin{table}[h!]\renewcommand{\arraystretch}{1.2}\centering \begin{tabular}{M{0.5cm}*{9}{r@{\hspace{0.5cm}}}r@{}}\toprule\multirow{2}{*}{\textbf{n}} & \multicolumn{8}{c}{\textbf{K}} \\& \textbf{2} & \textbf{3} &\textbf{4} &\textbf{5} &\textbf{6} &\textbf{7} &\textbf{8} &\textbf{9} &\textbf{10} \tabularnewline\hline
\textbf{10} & 
7,5 & 9,3 & 8,1 & 7,1 & 3,5 & 1,6 & 0,5 & 0,0 & \\
\textbf{11} & 
9,6 & 10,9 & 10,6 & 8,2 & 5,3 & 2,9 & 1,6 & 0,8 & 0,0 & \\
\textbf{12} & 
11,4 & 13,2 & 12,6 & 10,6 & 7,7 & 4,2 & 2,8 & 1,1 & 0,1 & \\
\textbf{13} & 
12,8 & 15,3 & 15,6 & 13,5 & 9,5 & 6,2 & 3,4 & 2,1 & 0,5 & \\
\textbf{14} & 
14,9 & 17,0 & 17,8 & 15,7 & 12,8 & 9,4 & 5,7 & 3,1 & 1,8 & \\
\textbf{15} & 
16,9 & 19,6 & 20,5 & 19,3 & 17,3 & 14,3 & 10,6 & 6,1 & 3,5 & \\
\textbf{16} & 
17,8 & 19,9 & 20,5 & 19,8 & 17,9 & 14,9 & 11,3 & 8,2 & 5,0 & \\
\textbf{17} & 
19,1 & 21,5 & 22,1 & 21,3 & 20,2 & 17,3 & 14,0 & 10,6 & 7,3 & \\
\textbf{18} & 
21,1 & 23,3 & 23,8 & 23,2 & 21,3 & 18,7 & 15,6 & 12,1 & 8,4 & \\
\textbf{19} & 
22,5 & 25,0 & 26,0 & 26,2 & 25,7 & 23,5 & 20,1 & 16,5 & 12,1 & \\
\textbf{20} & 
25,1 & 28,6 & 30,1 & 29,6 & 28,4 & 26,1 & 22,8 & 18,9 & 14,4 & \\
\bottomrule
\end{tabular}
\caption{ Mean gain in percentage over the instances of $D_1$ when adding 
  2-partition inequalities.}
\label{tab:st_graph1}
\end{table}

\begin{table}[h!]\renewcommand{\arraystretch}{1.2}\centering \begin{tabular}{M{0.5cm}*{8}{r@{\hspace{0.5cm}}}r@{\hspace{0.2cm}}}\toprule\multirow{2}{*}{\textbf{n}} & \multicolumn{8}{c}{\textbf{K}} \\& \textbf{2} & \textbf{3} &\textbf{4} &\textbf{5} &\textbf{6} &\textbf{7} &\textbf{8} &\textbf{9} &\textbf{10} \tabularnewline\hline
\textbf{10} & 
298,0 & 153,6 & 73,7 & 32,2 & 15,9 & 4,3 & 0,4 & 0,0  \\
\textbf{11} & 
368,7 & 197,8 & 104,3 & 49,6 & 17,8 & 7,5 & 3,8 & 0,8 & 0,0  \\
\textbf{12} & 
434,0 & 239,2 & 135,5 & 73,0 & 31,2 & 12,4 & 5,4 & 2,3 & 0,6  \\
\textbf{13} & 
505,7 & 288,4 & 171,5 & 98,5 & 49,2 & 19,9 & 9,1 & 5,1 & 2,0  \\
\textbf{14} & 
593,6 & 343,6 & 206,9 & 119,2 & 68,2 & 33,2 & 14,7 & 6,3 & 3,4  \\
\textbf{15} & 
673,2 & 395,4 & 244,7 & 152,0 & 95,5 & 54,2 & 25,5 & 13,6 & 6,1  \\
\textbf{16} & 
782,9 & 465,0 & 294,0 & 181,8 & 112,1 & 68,1 & 35,4 & 17,6 & 9,8  \\
\textbf{17} & 
855,6 & 515,5 & 328,1 & 214,8 & 137,8 & 84,9 & 48,2 & 23,7 & 13,4  \\
\textbf{18} & 
964,1 & 582,2 & 376,3 & 247,7 & 159,8 & 102,0 & 63,6 & 36,9 & 19,6  \\
\textbf{19} & 
1041,6 & 637,4 & 421,9 & 289,1 & 199,5 & 134,0 & 88,2 & 52,4 & 26,8  \\
\textbf{20} & 
1186,6 & 739,3 & 491,1 & 336,3 & 228,9 & 153,1 & 103,1 & 63,5 & 35,9  \\
\bottomrule
\end{tabular}
\caption{Mean gain in percentage over the instances of $D_1$ when adding 
  general clique inequalities.}
\label{tab:dep_graph1}
\end{table}
\bigskip

In the case of the  instances of $D_2$, the 2-chorded cycle inequality
still provide low  gains as represented in table~\ref{tab:tcc_graph2}.
For a fix  value of $n$ we observe that the  mean improvement does not
vary depending  on $K$ except for  the top right corner  of the table.
These couples  $(n,K)$ correspond to the instances  which may directly
be    solved   optimally   by    the   relaxation    (as   represented
table~\ref{tab:isRoot2}).

\begin{table}[h!]\renewcommand{\arraystretch}{1.2}\centering \begin{tabular}{M{0.5cm}*{9}{r@{\hspace{0.5cm}}}r@{}}\toprule\multirow{2}{*}{\textbf{n}} & \multicolumn{8}{c}{\textbf{K}} \\& \textbf{2} & \textbf{3} &\textbf{4} &\textbf{5} &\textbf{6} &\textbf{7} &\textbf{8} &\textbf{9} &\textbf{10} \tabularnewline\hline
\textbf{10} & 
0,4 & 0,4 & 0,4 & 0,6 & 0,8 & 0,8 & 0,4 & 0,0 & \\
\textbf{11} & 
0,4 & 0,4 & 0,4 & 0,6 & 1,0 & 1,5 & 1,2 & 0,5 & 0,0  & \\
\textbf{12} & 
0,9 & 0,7 & 0,7 & 0,9 & 1,3 & 2,1 & 2,5 & 1,7 & 0,7 & \\
\textbf{13} & 
1,0 & 1,0 & 1,0 & 1,1 & 1,3 & 1,9 & 2,7 & 2,8 & 1,7 & \\
\textbf{14} & 
1,2 & 1,2 & 1,2 & 1,3 & 1,5 & 1,9 & 2,9 & 3,5 & 2,8 & \\
\textbf{15} & 
2,1 & 2,1 & 2,2 & 2,3 & 2,5 & 2,7 & 3,5 & 4,2 & 4,4 & \\
\textbf{16} & 
2,6 & 2,7 & 2,7 & 2,7 & 2,8 & 3,1 & 3,6 & 4,7 & 5,2 & \\
\textbf{17} & 
3,1 & 3,2 & 3,2 & 3,3 & 3,3 & 3,4 & 3,7 & 4,5 & 5,5 & \\
\textbf{18} & 
4,0 & 4,0 & 4,0 & 4,0 & 4,1 & 4,2 & 4,3 & 4,8 & 5,9 & \\
\textbf{19} & 
4,5 & 4,8 & 4,9 & 5,0 & 5,0 & 5,0 & 5,0 & 5,2 & 5,9 & \\
\textbf{20} & 
5,5 & 5,8 & 5,8 & 5,8 & 5,8 & 5,8 & 5,9 & 6,0 & 6,4 & \\

\bottomrule
\end{tabular}
\caption{Mean gain in percentage over the instances of $D_2$ when adding 
  2-chorded cycle inequalities.}
\label{tab:tcc_graph2}
\end{table}

The   results  for  the   2-partition  inequalities,   represented  in
Figure~\ref{tab:st_graph2}, are similar. Although lower
than for $D_1$, the mean gains are twice as high as those given
by the 2-chorded cycle inequalities.

\begin{table}[h!]\renewcommand{\arraystretch}{1.2}\centering \begin{tabular}{M{0.5cm}*{9}{r@{\hspace{0.5cm}}}r@{}}\toprule\multirow{2}{*}{\textbf{n}} & \multicolumn{8}{c}{\textbf{K}} \\& \textbf{2} & \textbf{3} &\textbf{4} &\textbf{5} &\textbf{6} &\textbf{7} &\textbf{8} &\textbf{9} &\textbf{10} \tabularnewline\hline
\textbf{10} & 
1,8 & 1,4 & 1,2 & 2,2 & 3,8 & 6,4 & 10,1 & 13,5 & \\
\textbf{11} & 
2,3 & 1,9 & 1,8 & 2,6 & 4,0 & 6,9 & 9,1 & 11,7 & 13,2 & \\
\textbf{12} & 
3,0 & 2,3 & 2,2 & 2,9 & 4,4 & 6,6 & 9,5 & 11,8 & 13,2 & \\
\textbf{13} & 
3,7 & 3,4 & 3,5 & 3,9 & 4,8 & 6,6 & 9,6 & 12,5 & 14,4 & \\
\textbf{14} & 
4,3 & 4,0 & 3,9 & 4,2 & 5,0 & 6,5 & 8,5 & 11,1 & 13,1 & \\
\textbf{15} & 
5,3 & 5,0 & 5,1 & 5,3 & 5,9 & 6,8 & 8,8 & 11,4 & 14,0 & \\
\textbf{16} & 
5,6 & 5,5 & 5,5 & 5,6 & 6,1 & 7,0 & 8,3 & 11,0 & 13,1 & \\
\textbf{17} & 
7,3 & 7,1 & 7,3 & 7,3 & 7,6 & 8,1 & 8,8 & 10,5 & 13,1 & \\
\textbf{18} & 
8,0 & 7,7 & 7,9 & 7,9 & 8,2 & 8,3 & 9,0 & 10,3 & 12,5 & \\
\textbf{19} & 
9,2 & 9,5 & 9,6 & 9,6 & 9,9 & 9,8 & 10,3 & 11,1 & 12,5 & \\
\textbf{20} & 
10,5 & 10,6 & 10,6 & 10,6 & 10,8 & 10,7 & 10,9 & 11,7 & 12,7 & \\

\bottomrule
\end{tabular}
\caption{Mean gain in percentage over the instances of $D_2$ when adding 
  2-partition inequalities.}
\label{tab:st_graph2}
\end{table}

Table~\ref{tab:dep_graph2} shows that  the gain is significantly lower
for the  general clique  inequalities. These inequalities  require the
presence of  a minimum  number of  edges from a  clique $E(Z)$  in the
$K$-partition. Since  the we minimize the  weight of the  edges in the
$K$ sets,  an instance of  $D_1$ will tend  to have less edges  in the
sets than  an instance of $D_2$. Nevertheless,  these inequalities can
still be useful for the small values of $K$.

\begin{table}[h!]\renewcommand{\arraystretch}{1.2}\centering \begin{tabular}{M{0.5cm}*{9}{r@{\hspace{0.5cm}}}r@{}}\toprule\multirow{2}{*}{\textbf{n}} & \multicolumn{8}{c}{\textbf{K}} \\& \textbf{2} & \textbf{3} &\textbf{4} &\textbf{5} &\textbf{6} &\textbf{7} &\textbf{8} &\textbf{9} &\textbf{10} \tabularnewline\hline
\textbf{10} & 
16,7 & 2,6 & 0,2 & 0,0 & 0,0 & 0,0 & 0,0 & 0,0 & \\
\textbf{11} & 
16,2 & 2,7 & 0,4 & 0,0 & 0,0 & 0,0 & 0,0 & 0,0 & 0,0  & \\
\textbf{12} & 
17,5 & 3,2 & 0,5 & 0,1 & 0,0 & 0,0 & 0,0 & 0,0 & 0,0 & \\
\textbf{13} & 
17,2 & 4,0 & 0,8 & 0,1 & 0,0 & 0,0 & 0,0 & 0,0 & 0,0 & \\
\textbf{14} & 
17,3 & 4,3 & 0,8 & 0,1 & 0,0 & 0,0 & 0,0 & 0,0 & 0,0 & \\
\textbf{15} & 
19,5 & 5,7 & 1,2 & 0,1 & 0,0 & 0,0 & 0,0 & 0,0 & 0,0 & \\
\textbf{16} & 
20,2 & 6,1 & 1,6 & 0,2 & 0,0 & 0,0 & 0,0 & 0,0 & 0,0 & \\
\textbf{17} & 
22,5 & 7,7 & 2,1 & 0,3 & 0,0 & 0,0 & 0,0 & 0,0 & 0,0 & \\
\textbf{18} & 
23,6 & 8,1 & 2,5 & 0,4 & 0,0 & 0,0 & 0,0 & 0,0 & 0,0 & \\
\textbf{19} & 
24,8 & 9,2 & 2,9 & 0,5 & 0,0 & 0,0 & 0,0 & 0,0 & 0,0 & \\
\textbf{20} & 
24,9 & 9,8 & 3,2 & 0,5 & 0,0 & 0,0 & 0,0 & 0,0 & 0,0 & \\

\bottomrule
\end{tabular}
\caption{Mean gain in percentage over the instances of $D_2$ when adding 
  general clique inequalities.}
\label{tab:dep_graph2}
\end{table}

As for the paw inequalities (table~\ref{tab:new_graph2}), they seem to
complement the  general clique  inequalities well: the  gain increases
with $K$.

\begin{table}[h!]\renewcommand{\arraystretch}{1.2}\centering \begin{tabular}{M{0.5cm}*{9}{r@{\hspace{0.5cm}}}r@{}}\toprule\multirow{2}{*}{\textbf{n}} & \multicolumn{8}{c}{\textbf{K}} \\& \textbf{2} & \textbf{3} &\textbf{4} &\textbf{5} &\textbf{6} &\textbf{7} &\textbf{8} &\textbf{9} &\textbf{10} \tabularnewline\hline
\textbf{10} & 
0,0 & 0,1 & 1,4 & 5,2 & 11,0 & 18,5 & 26,5 & 35,3 & \\
\textbf{11} & 
0,0  & 0,1 & 1,2 & 4,4 & 9,7 & 16,7 & 23,8 & 30,4 & 37,3 & \\
\textbf{12} & 
0,0  & 0,0 & 0,6 & 2,9 & 7,7 & 14,4 & 21,7 & 27,8 & 33,4 & \\
\textbf{13} & 
0,0 & 0,1 & 0,6 & 2,3 & 6,1 & 12,2 & 19,3 & 26,3 & 32,2 & \\
\textbf{14} & 
0,0  & 0,0 & 0,3 & 1,5 & 4,5 & 9,9 & 16,5 & 22,8 & 28,5 & \\
\textbf{15} & 
0,0 & 0,0 & 0,1 & 0,8 & 3,0 & 7,6 & 13,9 & 20,6 & 26,7 & \\
\textbf{16} & 
0,0  & 0,0 & 0,1 & 0,6 & 2,2 & 6,1 & 11,8 & 18,5 & 24,7 & \\
\textbf{17} & 
0,0 & 0,0 & 0,1 & 0,4 & 1,4 & 4,1 & 9,0 & 15,3 & 21,8 & \\
\textbf{18} & 
0,0 & 0,0 & 0,1 & 0,2 & 0,8 & 2,8 & 7,2 & 13,1 & 19,6 & \\
\textbf{19} & 
0,0 & 0,0 & 0,0 & 0,1 & 0,4 & 1,6 & 5,0 & 10,1 & 16,2 & \\
\textbf{20} & 
0,0 & 0,0 & 0,0 & 0,1 & 0,2 & 1,0 & 3,8 & 8,6 & 14,6 & \\

\bottomrule
\end{tabular}
\caption{Mean gain in percentage over the instances of $D_2$ when adding 
  paw inequalities.}
\label{tab:new_graph2}
\end{table}

When instances with only negative weights are considered, we
fail  at   finding  any  violated   inequality  except  for   the  paw
inequalities  (\ref{tab:new_graph3}).   These  instances  are  however
easier to solve in practice.

\begin{table}[h!]\renewcommand{\arraystretch}{1.2}\centering \begin{tabular}{M{0.5cm}*{9}{r@{\hspace{0.5cm}}}r@{}}\toprule\multirow{2}{*}{\textbf{n}} & \multicolumn{8}{c}{\textbf{K}} \\& \textbf{2} & \textbf{3} &\textbf{4} &\textbf{5} &\textbf{6} &\textbf{7} &\textbf{8} &\textbf{9} &\textbf{10} \tabularnewline\hline
\textbf{10} & 
0,7 & 1,5 & 2,5 & 3,8 & 5,2 & 9,6 & 20,3 & 32,0 & \\
\textbf{11} & 
0,7 & 1,4 & 2,4 & 3,6 & 5,0 & 6,5 & 14,2 & 24,3 & 33,3 & \\
\textbf{12} & 
0,6 & 1,4 & 2,2 & 3,3 & 4,6 & 6,1 & 9,9 & 18,2 & 27,2 & \\
\textbf{13} & 
0,6 & 1,4 & 2,3 & 3,3 & 4,6 & 6,0 & 7,6 & 13,6 & 21,5 & \\
\textbf{14} & 
0,6 & 1,3 & 2,1 & 3,1 & 4,2 & 5,5 & 7,0 & 10,2 & 16,6 & \\
\textbf{15} & 
0,6 & 1,3 & 2,1 & 3,0 & 4,1 & 5,3 & 6,8 & 8,4 & 13,2 & \\
\textbf{16} & 
0,6 & 1,3 & 2,0 & 2,9 & 3,9 & 5,1 & 6,4 & 7,9 & 10,5 & \\
\textbf{17} & 
0,6 & 1,3 & 2,0 & 2,9 & 3,8 & 4,9 & 6,2 & 7,7 & 9,2 & \\
\textbf{18} & 
0,6 & 1,2 & 1,9 & 2,7 & 3,6 & 4,7 & 5,8 & 7,1 & 8,6 & \\
\textbf{19} & 
0,6 & 1,2 & 1,9 & 2,7 & 3,5 & 4,5 & 5,6 & 6,8 & 8,2 & \\
\textbf{20} & 
0,6 & 1,2 & 1,8 & 2,6 & 3,4 & 4,3 & 5,3 & 6,5 & 7,7 & \\

\bottomrule
\end{tabular}
\caption{Mean gain in percentage over the instances of $D_3$ when adding 
  paw inequalities.}
\label{tab:new_graph3}
\end{table}
\bigskip

\noindent\textbf{Conclusion}

We  have   introduced  a   new  formulation  for   the  K-partitioning
problem. Thanks  to the addition  of representative variables,  we are
able  to break  the symmetry  in the  edge variable  formulation.  The
resulting formulation  shows to be stronger than  the formulation with
node-cluster  variables and  edge  variables used  by several  authors
(\cite{chopra1993partition,kaibel2011orbitopal,fan2010linear})  when K
is greater  than $2$, at  least on complete  graphs.  We have  proved in
this paper  facet-defining results  for several classical  families of
inequalities, and  for a new family  of inequalities that  seems to be
useful when there are negative edges.

The computing time for the 20-node  instances is only of a few minutes
using CPLEX 12.5 on a  desktop computer. To actually solve problems to
optimality for  higher values  of $n$ will  need to find  a compromise
between the separation  and the solving of the  linear programs at the
nodes of a branch and bound  procedure. Still the results of this work
are promising and show the interest of the polyhedral approach.

Further work  will concentrate on improving  the separation procedures
and developing  a branch  and cut framework  for the  application that
motivated this study \cite{alesmethodology}.

\section{References}
 \bibliographystyle{elsarticle-num} 
  \bibliography{bibliography}

\end{document}